\documentclass[letterpaper,USenglish]{lipics-v2019}

\usepackage[ruled]{algorithm}
\usepackage{amstext,algpseudocode}
\usepackage{mdframed}
\hideLIPIcs
\nolinenumbers

\newcounter{questioncounter}
\newtheorem{problem}[questioncounter]{Open Problem}

\newtheorem*{mainproblem}{Main Problem}
\newenvironment{mproblem}
  {\begin{mdframed}\begin{mainproblem}}
  {\end{mainproblem}\end{mdframed}}

\newcommand{\reals}{\mathbb{R}}

\newcommand{\q}{\mathbb{Q}}

\newcommand{\cm}[1]{\operatorname{CM}_{#1}}

\newcommand{\ideal}[1]{\left\langle #1 \right\rangle}
\newcommand{\grobner}{Gr\"obner}
\newcommand{\ffield}[1]{\operatorname{Frac}{#1}}
\newcommand{\codim}[1]{\operatorname{codim}{#1}}

\newcommand{\supp}[1]{\operatorname{supp}{#1}}
\newcommand{\res}[3]{\operatorname{Res}(#1, #2, #3)}

\newcommand{\lmat}[1]{\operatorname{\mathcal L}_{#1}} 
\newcommand{\amat}[1]{\operatorname{\mathcal A}(#1)} 
\newcommand{\smat}[1]{\operatorname{\mathcal S}_{#1}} 
\newcommand{\cres}[3]{\operatorname{CRes}(#1, #2, #3)}

\newcommand{\rank}{\operatorname{rank}}

\title{Combinatorial Resultants in the Algebraic Rigidity Matroid}

\titlerunning{Combinatorial Resultants in the Algebraic Rigidity Matroid}

\author{Goran Mali\'c}{Computer Science Department, Smith College, Northampton MA, USA 
\and 
\url{http://goranmalic.com} }
{goranm00@gmail.com, gmalic@smith.edu}
{https://orcid.org/0000-0001-9340-7827}
{}

\author{Ileana Streinu\footnote{Corresponding author}}{Computer Science Department, Smith College, Northampton, MA, USA 
\and 
\url{http://cs.smith.edu/~istreinu}}
{istreinu@smith.edu, streinu@cs.umass.edu}
{https://orcid.org/0000-0002-1825-0097}
{}

\authorrunning{G. Mali\'c and I. Streinu} 

\Copyright{Goran Mali\'c and Ileana Streinu}

\ccsdesc[100]{General and reference~Performance}
\ccsdesc[100]{General and reference~Experimentation}
\ccsdesc[500]{Theory of computation~Computational geometry}
\ccsdesc[500]{Mathematics of computing~Matroids and greedoids}
\ccsdesc[500]{Mathematics of computing~Mathematical software performance}
\ccsdesc[300]{Computing methodologies~Combinatorial algorithms}
\ccsdesc[300]{Computing methodologies~Algebraic algorithms}

\keywords{Cayley-Menger ideal, rigidity matroid, circuit polynomial, combinatorial resultant, inductive construction, Gr\"obner basis elimination} 

\relatedversion{This is the full paper accompanying the extended abstract \cite{malic:streinu:socg}} 

\funding{Both authors acknowledge funding from the NSF CCF:1703765 grant to Ileana Streinu.}

\begin{document}

\maketitle

\begin{abstract}	
Motivated by a rigidity-theoretic perspective on the {\em Localization Problem} in 2D, we develop an algorithm for computing {\em circuit polynomials} in the algebraic rigidity matroid $\cm{n}$ associated to the Cayley-Menger ideal for $n$ points in 2D.  We introduce {\em combinatorial resultants}, a new operation on graphs that captures properties of the Sylvester resultant of two polynomials in the algebraic rigidity matroid. We show that every rigidity circuit has a {\em construction tree} from $K_4$ graphs based on this operation. Our algorithm performs an {\em algebraic elimination} guided by the construction tree, and uses classical resultants, factorization and ideal membership. 
To demonstrate its effectiveness, we implemented our algorithm in Mathematica: it took less than 15 seconds on an example where a Gr\"obner Basis calculation took 5 days and 6 hrs.
\end{abstract}

\parskip7pt
\parindent0pt

\section{Introduction}
\label{app:sec:introduction}

This paper addresses combinatorial, algebraic and algorithmic aspects of a question motivated by the following ubiquitous problem from {\em distance geometry}: 

\subparagraph{Localization.} A graph together with {\em weights} associated to its edges is given. The goal is to find {\em placements} of the graph in some Euclidean space, so that the edge lengths match the given weights. In this paper, we work in 2D. A system of quadratic equations can be easily set up so that the possible placements are among the (real) solutions of this system. Rigidity Theory can help predict, {\em a priori,} whether the set of solutions will be discrete (if the given weighted graph is {\em rigid}) or continuous (if the graph is {\em flexible}).  In the rigid case, the  double-exponential Gr\"obner basis algorithm can be used, in principle,  to eliminate all but one of the variables. Once a polynomial in a single variable is obtained, numerical methods are used to solve it. We then select one solution, substitute it in the original equations, eliminate to get a polynomial in a new variable and repeat. 

\subparagraph{Single Unknown Distance Problem.}  Instead of attempting to directly compute the coordinates of all the vertices, we restrict our attention to the related problem of finding the possible values of a {\em single unknown distance} corresponding to a {\em non-edge} (a pair of vertices that are not connected by an edge). Indeed, {\em if} we could solve this problem for a collection of non-edge pairs that form a trilateration when added to the original collection of edges, {\em then} a single solution in Cartesian coordinates could be easily computed afterwards in linearly many steps of quadratic equation solving.

\subparagraph{Rigidity circuits.}
We formulate the {\em single unknown distance} problem in terms of Cayley rather than Cartesian coordinates. Known theorems from Distance Geometry, Rigidity Theory and Matroid Theory help reduce this problem to finding a certain irreducible polynomial in the Cayley-Menger ideal, called the  {\em circuit polynomial}. Its support is a graph called a {\em circuit in the rigidity matroid}, or shortly a {\em rigidity circuit.} Substituting given edge lengths in the circuit polynomial results in a uni-variate polynomial which can be solved for the unknown distance.

The focus of this paper is the following:
\begin{mproblem}
{\em Given a rigidity circuit, compute its corresponding circuit polynomial.}
\end{mproblem}

\subparagraph{Related work.}  
While both \emph{distance geometry} and \emph{rigidity theory} have a distinguished history for which a comprehensive overview would be too long to include here, very little is known about computing \emph{circuit polynomials.}
\emph{To the best of our knowledge,} their study in \emph{arbitrary} polynomial ideals was initiated in the PhD thesis of Rosen \cite{rosen:thesis}. His Macaulay2 code \cite{rosen:GitHubRepo} is useful for exploring small cases, but the Cayley-Menger ideal is beyond its reach. A recent article \cite{rosen:sidman:theran:algebraicMatroidsAction:2020} popularizes algebraic matroids and uses for illustration the smallest circuit polynomial $K_4$ in the Cayley-Menger ideal. \emph{We could not find non-trivial examples anywhere. } Indirectly related to our problem are results such as \cite{WalterHusty}, where an explicit univariate polynomial of degree 8 is computed (for an unknown angle in a $K_{3,3}$ configuration given by edge lengths, from which the placement of the vertices is determined) and \cite{sitharam:convexConfigSpaces:2010}, for its usage of Cayley coordinates in the study of configuration spaces of some families of distance graphs. A closely related problem is that of computing the \emph{number of embeddings of a minimally rigid graph} \cite{streinu:borcea:numberEmbeddings:2004}, which has received \emph{a lot} of attention in recent years (e.g. \cite{capco:schicho:realizations:2017,bartzos:emiris:etAl:realizations:2021,emiris:tsigaridas:varvitsiotis:mixedVolume,emiris:mourrain}, to name a few).
References to specific results in the literature that are relevant to the theory developed here and to our proofs are given throughout the paper.

\subparagraph{How tractable is the problem?} 
Circuit polynomial computations can be done, in principle, with the double-exponential time Gr\"obner basis algorithm with an elimination order.  On one example, the {\bf GroebnerBasis} function of Mathematica 12 (running on a 2019 iMac computer with 6 cores at 3.6Ghz) took 5 days and 6 hours, but in most cases it timed out or crashed. Our goal is to make such calculations {\em more tractable} by taking advantage of {\em structural information} inherent in the problem. 
 
\subparagraph{Our Results.} 
We describe a new {\em algorithm to compute a circuit polynomial with known support.} It relies on resultant-based elimination steps guided by a novel {\em inductive construction for rigidity circuits}. 
While inductive constructions have been often used in Rigidity Theory, most notably the Henneberg sequences for Laman graphs \cite{henneberg:graphischeStatik:1911-68} and Henneberg II sequences for rigidity circuits \cite{BergJordan}, we argue that our construction is more {\em natural} due to its direct algebraic interpretation. In fact, this paper originated from our attempt to interpret Henneberg II algebraically. We have implemented our method in Mathematica and applied it successfully to compute all circuit polynomials on up to $6$ vertices and a few on $7$ vertices, the largest of which having over two million terms. 
The previously mentioned example that took over 5 days to complete with GroebnerBasis, was solved by our algorithm in less than 15 seconds. 

\subparagraph{Main Theorems.} We first define the {\em combinatorial resultant} of two graphs as an abstraction of the classical resultant. Our main theoretical result is split into the combinatorial Theorem~\ref{thm:combResConstruction} and the algebraic Theorem~\ref{thm:circPolyConstruction}, each with an algorithmic counterpart. 


\begin{theorem}
	\label{thm:combResConstruction}
	Each rigidity circuit can be obtained, inductively, by applying combinatorial resultant operations starting from $K_4$ circuits. The construction is captured by a binary {\em resultant tree} whose nodes are intermediate rigidity circuits and whose leaves are $K_4$ graphs.
\end{theorem}

Theorem~\ref{thm:combResConstruction} leads to a  {\em graph algorithm} for finding a {\em combinatorial resultant tree} of a circuit. Each step of the construction can be carried out in polynomial time using variations on the {\em Pebble Game} matroidal sparsity algorithms \cite{streinu:lee:pebbleGames:2008} combined with Hopcroft and Tarjan's linear time $3$-connectivity algorithm \cite{hopcroft:tarjan:73}. However, it is conceivable that the tree could be exponentially large and thus the entire construction could take an exponential number of steps: understanding in detail the algorithmic complexity of our method {\em remains a problem for further investigation.}


\begin{theorem}
	\label{thm:circPolyConstruction}
	Each circuit polynomial can be obtained, inductively, by applying resultant operations. The procedure is guided by the combinatorial resultant tree from Theorem~\ref{thm:combResConstruction} and builds up from $K_4$ circuit polynomials. At each step, the resultant produces a polynomial that may not be irreducible. A polynomial factorization and a test of membership in the ideal is applied to identify the factor which is the circuit polynomial.
\end{theorem}

The resulting {\em algebraic elimination algorithm} runs in exponential time, in part because of the growth in size of the polynomials that are being produced. Several theoretical {\em open questions} remain, whose answers may affect the precise time complexity analysis. 

\subparagraph{Computational experiments.} We implemented our algorithms in Mathematica V12.1.1.0 on two personal computers with the following specifications: Intel i5-9300H 2.4GHz, 32 GB RAM, Windows 10 64-bit; and Intel i5-9600K 3.7GHz, 16 GB RAM, macOS Mojave 10.14.5. We also explored Macaulay2, but it was much slower 
than Mathematica (hours vs. seconds) in computing one of our examples. The polynomials resulting from our calculations are made available on a github repository \cite{malic:streinu:GitHubRepo}. 

\subparagraph{Overview of the paper.} In Section \ref{sec:prelimRigidity} we introduce the background concepts from matroid theory and rigidity theory. We introduce combinatorial resultants in Section \ref{sec:combRes}, prove Theorem \ref{thm:combResConstruction} and describe the algorithm for computing a {\em combinatorial resultant tree.} In Section~\ref{sec:prelimResultants}  we introduce the background concepts pertaining to resultants and elimination ideals. In Section~\ref{sec:prelimAlgMatroids} we introduce algebraic matroids. In Section~\ref{sec:prelimCMideal} we introduce the Cayley-Menger ideal and define properties of its circuit polynomials. In Section \ref{sec:algResCircuits} we prove Theorem \ref{thm:circPolyConstruction} and in Section~\ref{sec:experiments} we present a summary of the preliminary experimental results we carried with our implementation. We conclude in Section~\ref{sec:concludingRemarks} with a summary of remaining open questions. 

\section{Preliminaries: rigidity circuits}
\label{sec:prelimRigidity}

We start with the combinatorial aspects of our problem. In this section we review the essential notions and results from combinatorial rigidity theory of bar-and-joint frameworks in dimension $2$ that are relevant for our paper.

\subparagraph{Notation.} We work with (sub)graphs given by subsets $E$ of edges of the complete graph $K_n$ on vertices $[n]:=\{1,\cdots,n\}$. If $G$ is a (sub)graph, then $V(G)$, resp. $E(G)$ denote its vertex, resp. edge set. The support of $G$ is $E(G)$. The {\em vertex span} $V(E)$ of edges $E$ is the set of all edge-endpoint vertices. A subgraph $G$ is {\em spanning} if its edge set $E(G)$ spans $[n]$. The {\em neighbours} $N(v)$ of $v$ are the vertices adjacent to $v$ in $G$.

\subparagraph{Frameworks.} A {\em 2D bar-and-joint framework} is a pair $(G,p)$ of a graph $G=(V,E)$ whose vertices $V=[n]:= \{1,2,\dots,n\}$ are mapped to points $p=\{ p_1,\dots,p_n\}$ in $\mathbb R^2$ via the \emph{placement map} $p\colon V\to\mathbb R^2$ given by $i\mapsto p_i$. We view the edges as {\em rigid bars} and the vertices as {\em rotatable joints} which allow the framework to deform continuously as long as the bars retain their original lengths. The {\em realization space} of the framework is the set of all its possible placements in the plane with the same bar lengths. Two realizations are congruent if they are related by a planar isometry. The {\em configuration space} of the framework is made of congruence classes of realizations. The {\em deformation space} of a given framework is the particular connected component of the configuration space that contains it.

A framework is {\em rigid} if its deformation space consists in exactly one configuration, and {\em flexible} otherwise. Combinatorial rigidity theory of bar-and-joint frameworks seeks to understand the rigidity and flexibility of frameworks in terms of their underlying graphs. 

\subparagraph{Laman Graphs.} The following theorem relates rigidity theory of 2D bar-and-joint frameworks to graph sparsity.

\begin{theorem} [Laman's Theorem]
	\label{thm:laman}
	A bar-and-joint framework is {\em generically} minimally rigid in 2D iff its underlying graph $G=(V,E)$ has exactly $|E|=2|V|-3$ edges, and any proper subset $V'\subset V$ of vertices spans at most $2|V'|-3$ edges.
\end{theorem}

The {\em genericity} condition appearing in the statement of this theorem refers to the placements of the vertices. We'll introduce this concept rigorously in Section~\ref{sec:prelimCMideal}. For now we retain the most important consequence of the genericity condition, namely that {\em small perturbations of the vertex placements do not change the rigidity or flexibility properties of the framework.} This allows us to refer to the rigidity and flexibility of a (generic) framework solely in terms of its underlying graph. 

A graph satisfying the conditions of Laman's theorem is called a {\em Laman graph}. It is {\em minimally rigid} in the sense that it has just enough edges to be rigid: if one edge is removed, it becomes {\em flexible.} Adding extra edges to a Laman graph keeps it rigid, but the minimality is lost. Such graphs are said to be rigid and {\em overconstrained}. In short, for a graph to be rigid, its vertex set must span a Laman graph; otherwise the graph is flexible.

\subparagraph{Matroids. } A matroid is an abstraction capturing (in)dependence relations among collections of elements from a {\em ground set}, and is inspired by both {\em linear} dependencies (among, say, rows of a matrix) and by {\em algebraic} constraints imposed by algebraic equations on a collection of otherwise free variables. The standard way to specify a matroid is via its {\em independent sets}, which have to satisfy certain axioms \cite{Oxley:2011} (skipped here, since they are not relevant for our presentation). A {\em base} is a maximal independent set and a set which is not independent is said to be {\em dependent}. A minimal dependent set is called a {\em circuit}. Relevant for our purposes are the following general aspects: (a) (hereditary property) a subset of an independent set is also independent; (b) all bases have the same cardinality, called the {\em rank} of the matroid. Further properties will be introduced in context, as needed.

In this paper we encounter three types of matroids: a {\em graphic matroid}, defined on a ground set given by all the edges $E_n:=\{ij: 1\leq i < j \leq n\}$ of the complete graph $K_n$; this is the {\em $(2,3)$-sparsity matroid} or the {\em generic rigidity matroid} described below; a {\em linear matroid}, defined on an isomorphic set of {\em row vectors} of the {\em rigidity matrix} associated to a bar-and-joint framework; 
and an {\em algebraic matroid},  defined on an isomorphic ground set of variables $X_n:=\{x_{ij}: 1\leq i < j \leq n\}$; this is the {\em algebraic matroid associated to the Cayley-Menger ideal} and will be defined in Section~\ref{sec:prelimCMideal}. 

\subparagraph{The $(2,3)$-sparsity matroid: independent sets, bases, circuits.}
The $(2,3)$-sparse graphs on $n$ vertices form the collection of independent sets for a matroid $\smat{n}$ on the ground set $E$ of edges of the complete graph $K_n$ \cite{whiteley:Matroids:1996}, called the (generic) {\em 2D rigidity matroid}. The bases of the matroid $\smat n$ are the maximal independent sets, hence the Laman graphs. A set of edges which is not sparse is a {\em dependent} set. For instance, adding one edge to a Laman graph creates a dependent set of $2n-2$ edges, called a Laman-plus-one graph (Fig.~\ref{fig:lamanAndCircuits}). 

\begin{figure}[ht]
	\centering
		\includegraphics[width=.24\textwidth]{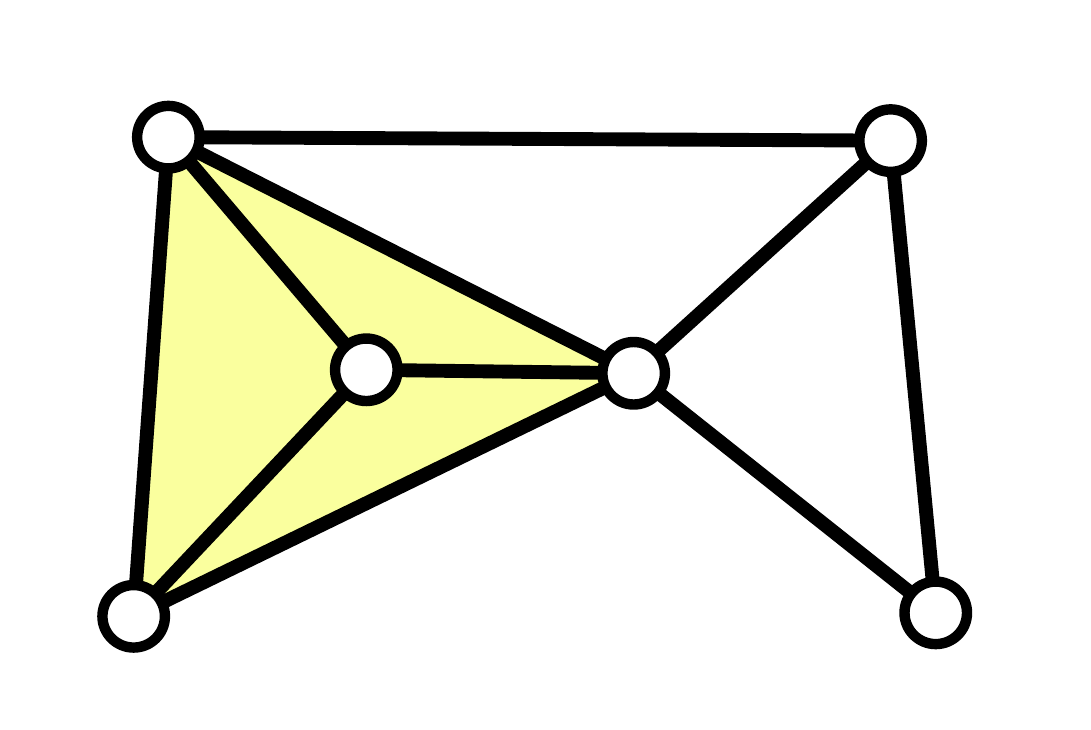}
		\includegraphics[width=.24\textwidth]{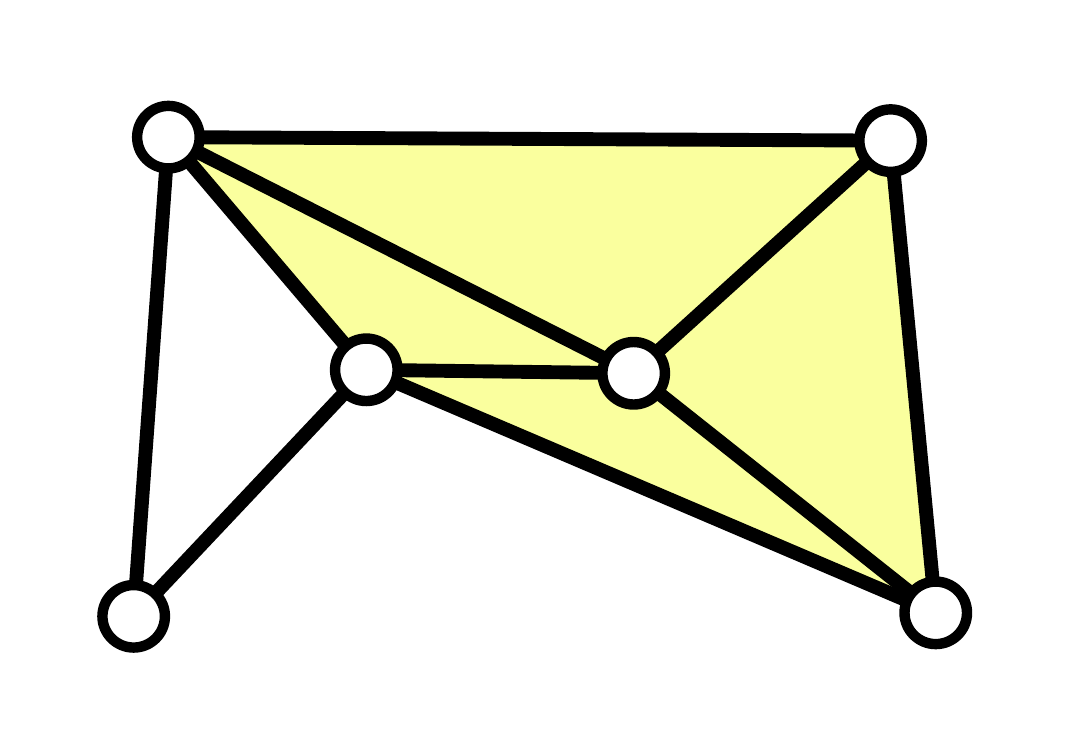}
		\includegraphics[width=.24\textwidth]{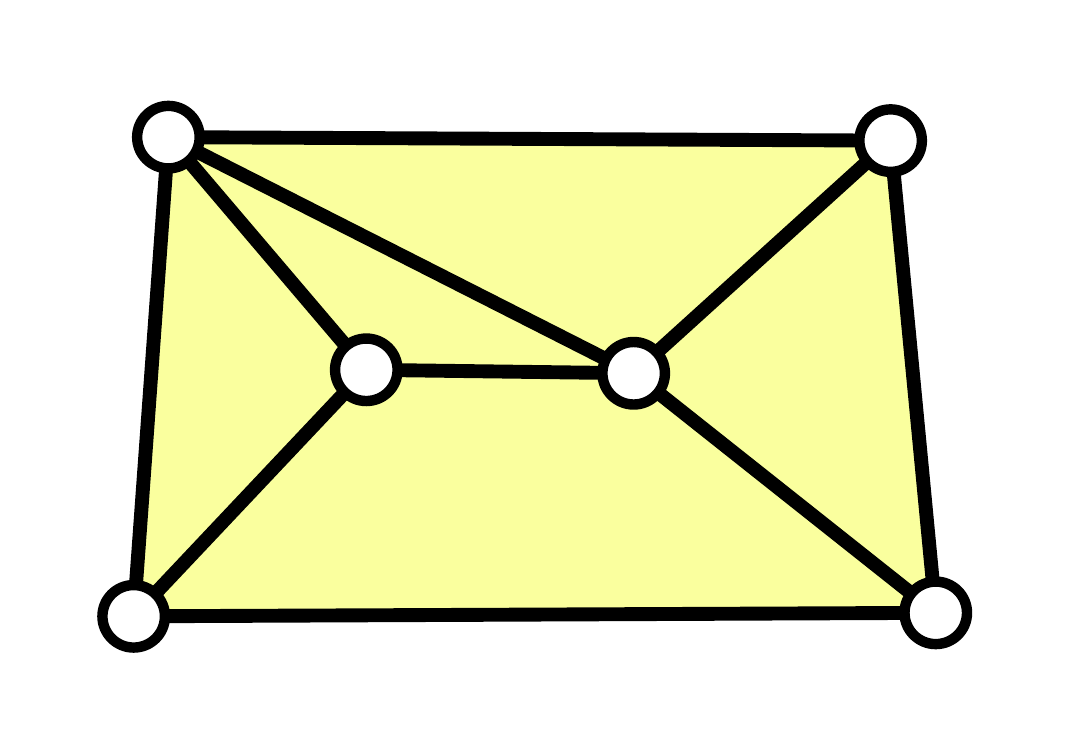}
	\caption{A Laman-plus-one graph contains a unique circuit (highlighted): (Left two) The circuit is not spanning the entire vertex set. (Right) A spanning circuit. 
	}
	\label{fig:lamanAndCircuits}
\end{figure}

A {\em minimal} dependent set is a (sparsity) {\em circuit}. The edges of a circuit span a subset of the vertices of $V$. A circuit spanning $V$ is said to be a {\em spanning} or {\em maximal} circuit in the sparsity matroid $\smat n$. See Fig.~\ref{fig:6circuits} for examples. 

\begin{figure}[ht]
	\centering
	\includegraphics[width=0.24\textwidth]{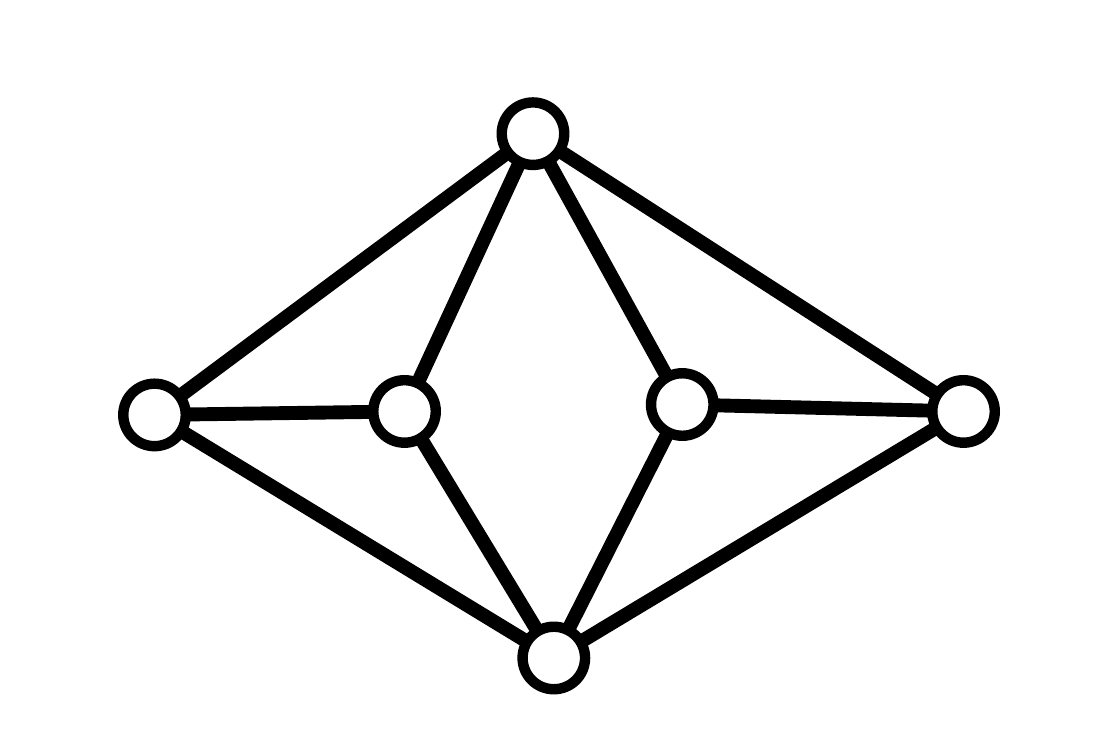}
	\includegraphics[width=0.24\textwidth]{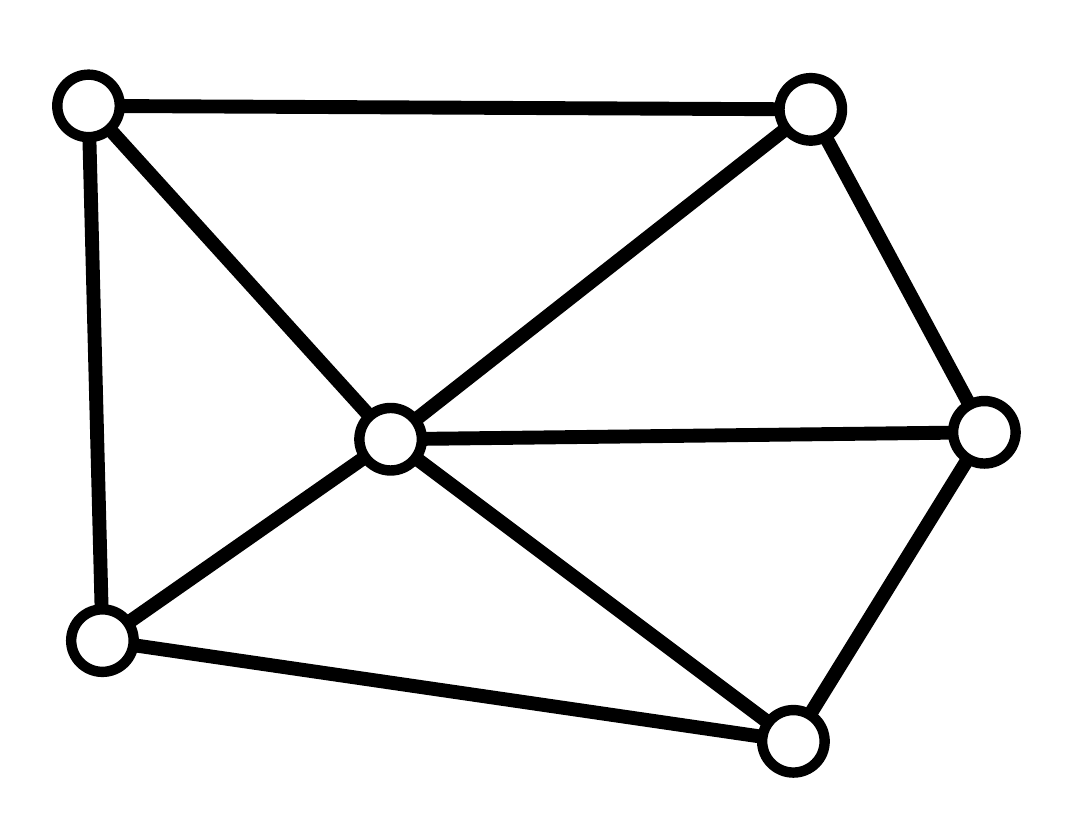}
	\includegraphics[width=0.24\textwidth]{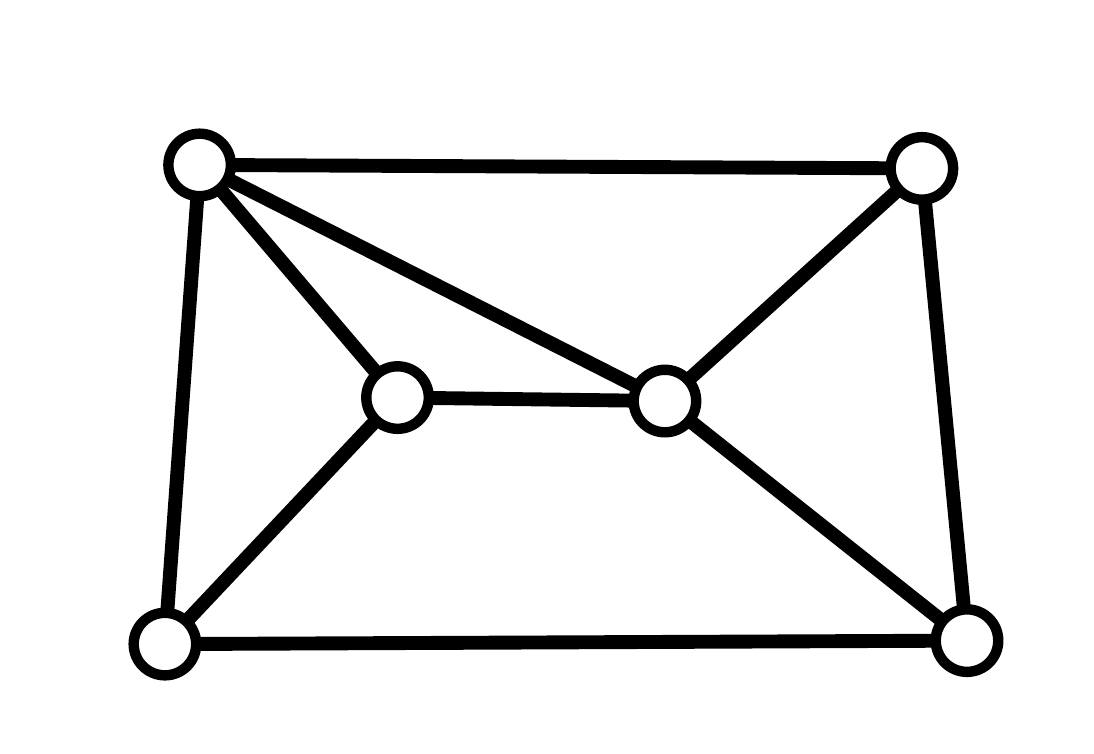}
	\includegraphics[width=0.24\textwidth]{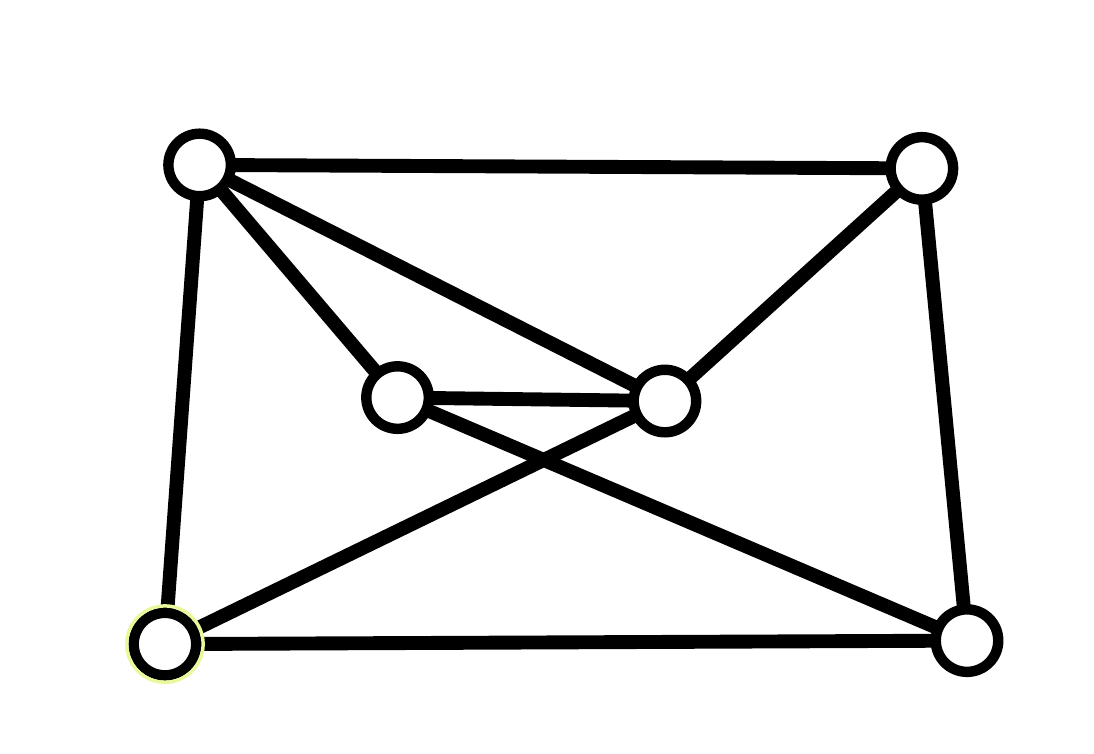}
	\caption{The four types of circuits on $n=6$ vertices: 2D {\em double-banana}, {\em $5$-wheel} $W_5$, {\em Desargues-plus-one} and $K_{3,3}$-plus-one.}
	\label{fig:6circuits}
\end{figure}

A {\em Laman-plus-one} graph contains a unique subgraph which is {\em minimally dependent}, in other words, a unique circuit. A spanning rigidity circuit $C=(V,E)$ is a special case of a Laman-plus-one graph: it has a total of $2n-2$ edges but it satisfies the $(2,3)$-sparsity condition on all proper subsets of at most $n'\leq n-1$ vertices. Simple sparsity considerations can be used to show that the removal of {\em any} edge from a circuit results in a Laman graph.

\begin{figure}[ht]
	\centering
	\includegraphics[width=0.3\textwidth]{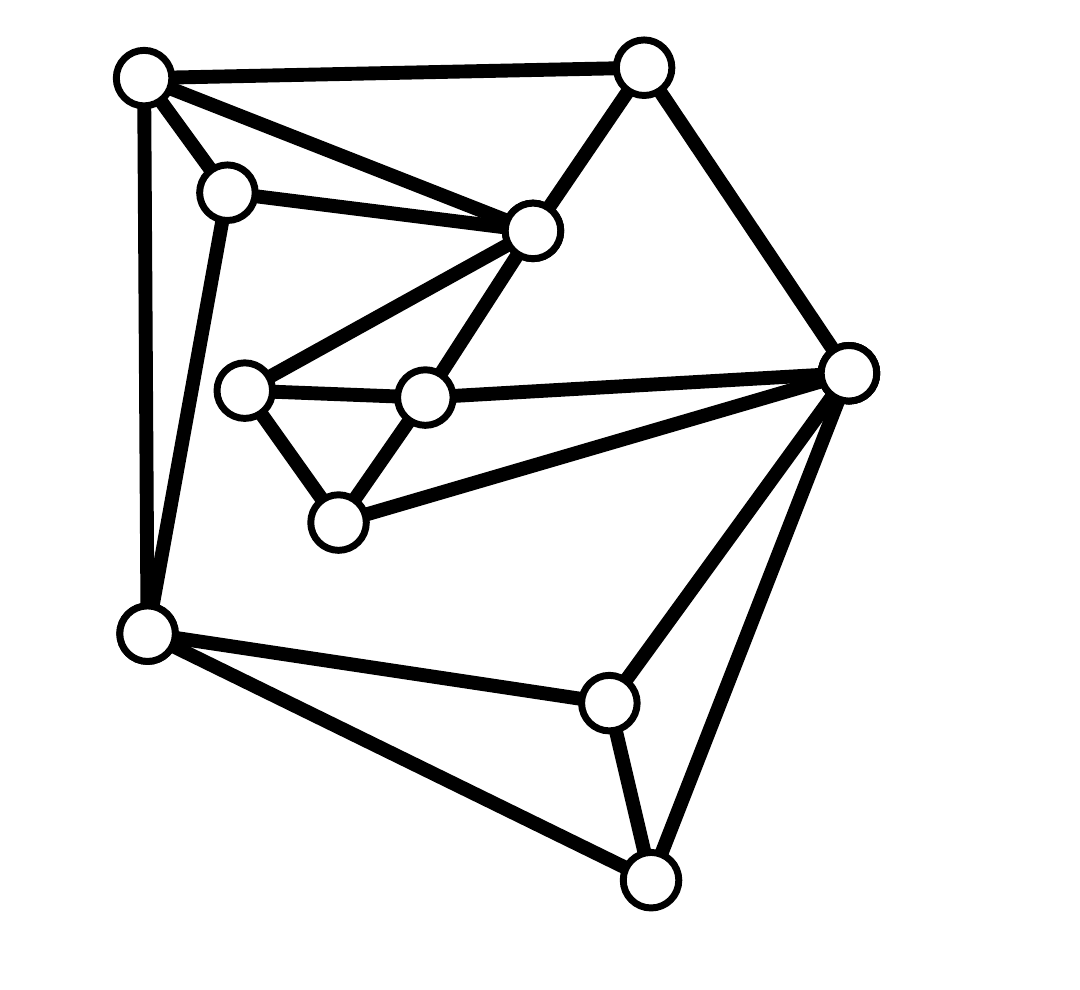}
	\includegraphics[width=0.3\textwidth]{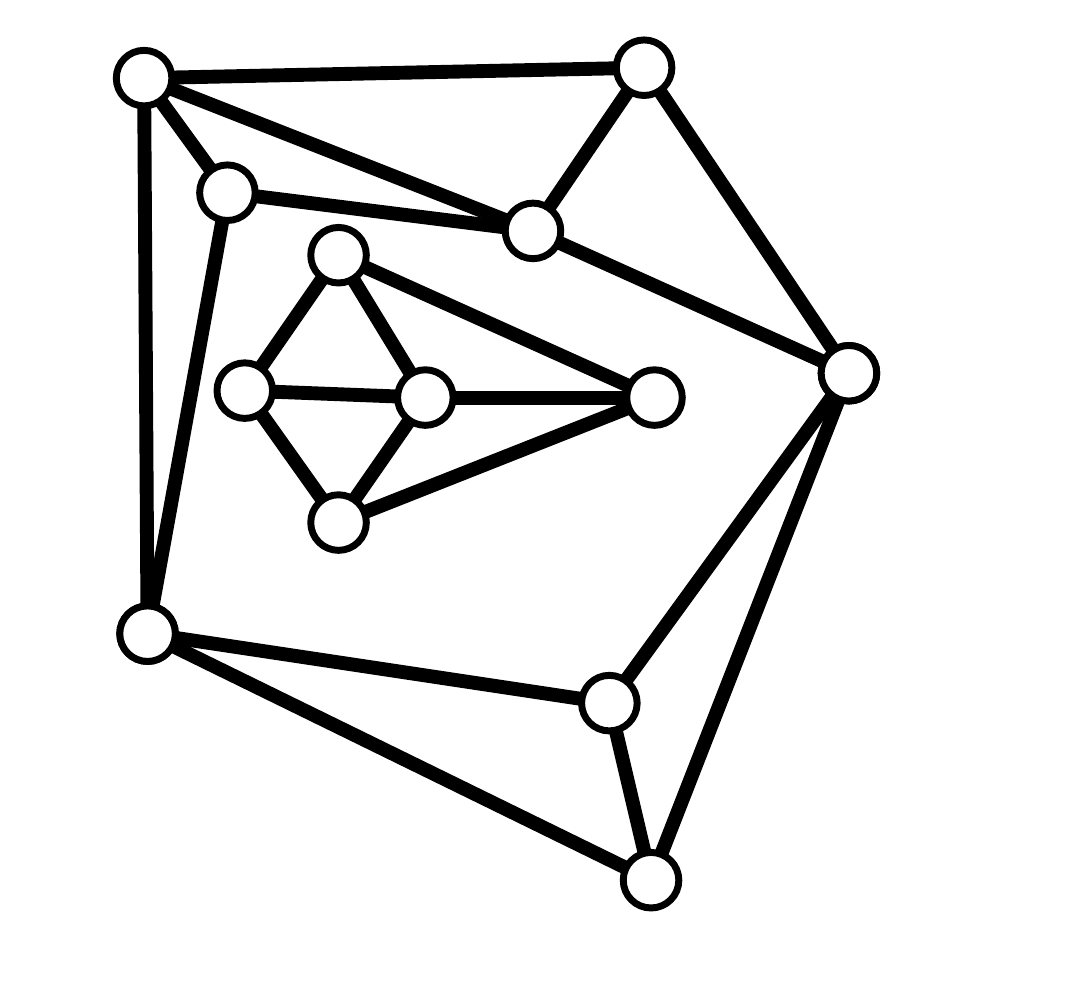}
	\includegraphics[width=0.3\textwidth]{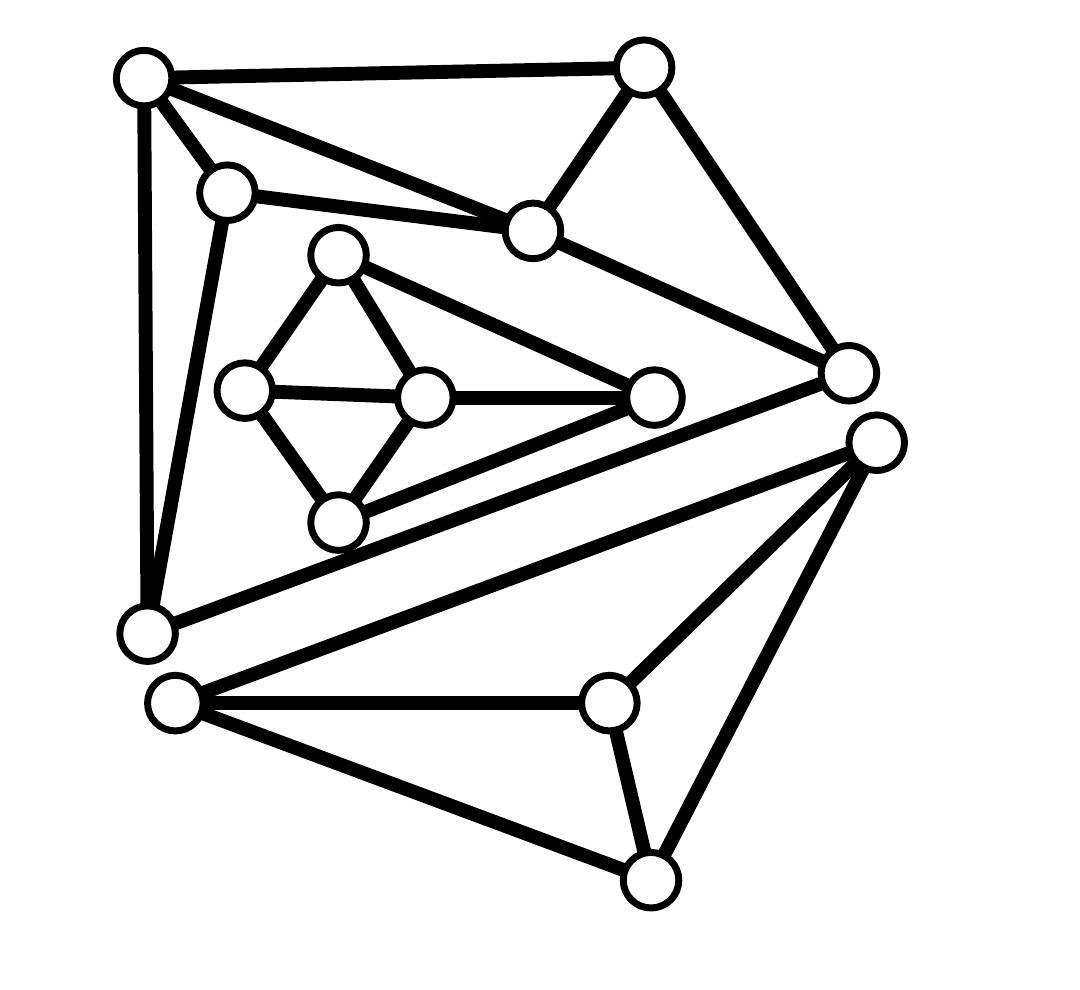}
	\caption{Splitting a $2$-connected circuit into $3$-connected circuits via inverse of two-sum operations.}
	\label{fig:2sum}
\end{figure}

\subparagraph{Operations on circuits.} If $G_1$ and $G_2$ are two graphs, we use a consistent notation for their number of vertices and edges $n_i=|V(G_i)|$, $m_i=|E(G_i)|$, $i=1,2$, and for their union and intersection of vertices and edges, as in $V_{\cup}=V(G_1)\cup V(G_2)$, $V_{\cap}=V(G_1)\cap V(G_2)$, $n_{\cup}=|V_{\cup}|$, $n_{\cap}=|V_{\cap}|$ and similarly for edges, with $m_{\cup}=|E_{\cup}|$ and $m_{\cap}=|E_{\cap}|$. 
Let $C_1$ and $C_2$ be two circuits with exactly one common edge $uv$. Their $2$-sum is the graph $C=(V,E)$ with $V=V_{\cup}$ and $E=E_{\cup} \setminus \{uv\}$. The inverse operation of splitting $C$ into $C_1$ and $C_2$ is called a $2$-split (Fig.~\ref{fig:2sum}).

\begin{lemma}
	\label{lem:twoSum}
	Any $2$-sum of two circuits is a circuit. Any 2-split of a circuit is a pair of circuits.
\end{lemma}

\begin{proof}\
From sparsity consideration, we prove that the $2$-sum has $2n-2$ edges (on $n$ vertices), and $(2,3)$-sparsity is also maintained on subsets. Indeed, the total number of edges in the $2$-sum is $m_1 + m_2 - 2 = 2n_1 - 2 + 2n_2 - 2 - 2 = 2(n_1 + n_2 -2) - 2 = 2 n - 2$. 

See also \cite{BergJordan}, Lemmas 4.1, 4.2.
\end{proof}

Note that not every circuit admits a $2$-separation, e.g.\ the Desargues-plus-one circuit does not admit a $2$-separation. However every circuit that is not $3$-connected does admit a $2$-separation (see also Lemma 2.4(c) in \cite{BergJordan}).

\subparagraph{Connectivity.} It is well known and easy to show that a circuit is always a $2$-connected graph. {\em If a circuit is not $3$-connected,}  we refer to it simply as a {\em $2$-connected circuit}.  The  Tutte decomposition \cite{tutte:connectivity:1966} of a $2$-connected graph into $3$-connected components amounts to identifying separating pairs of vertices. For a circuit, the separating pairs induce $2$-splits (inverse of $2$-sum) operations and produce smaller circuits. Thus a $2$-connected circuit can be constructed from $3$-connected circuits via $2$-sums (Fig.~\ref{fig:2sum}).

\begin{figure}[th]
	\centering
	\includegraphics[width=0.3\textwidth]{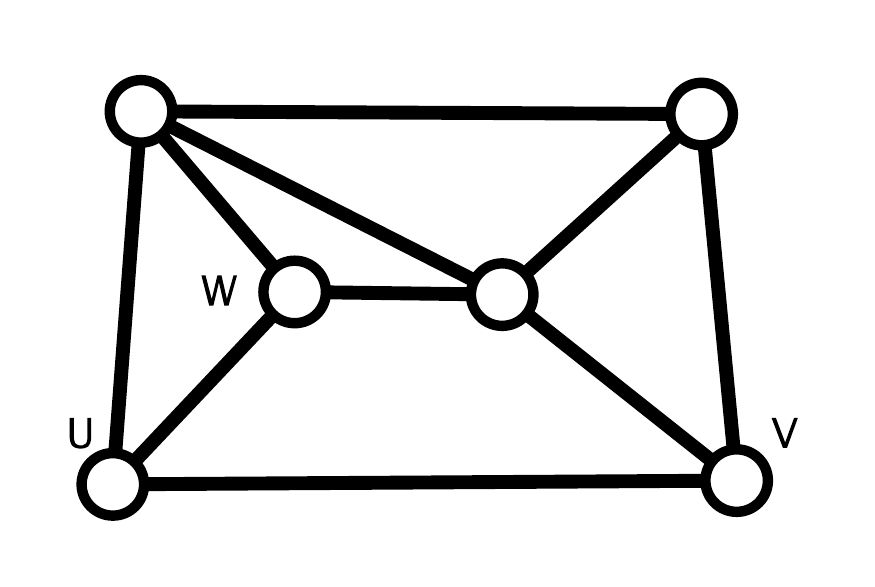}
	\includegraphics[width=0.33\textwidth]{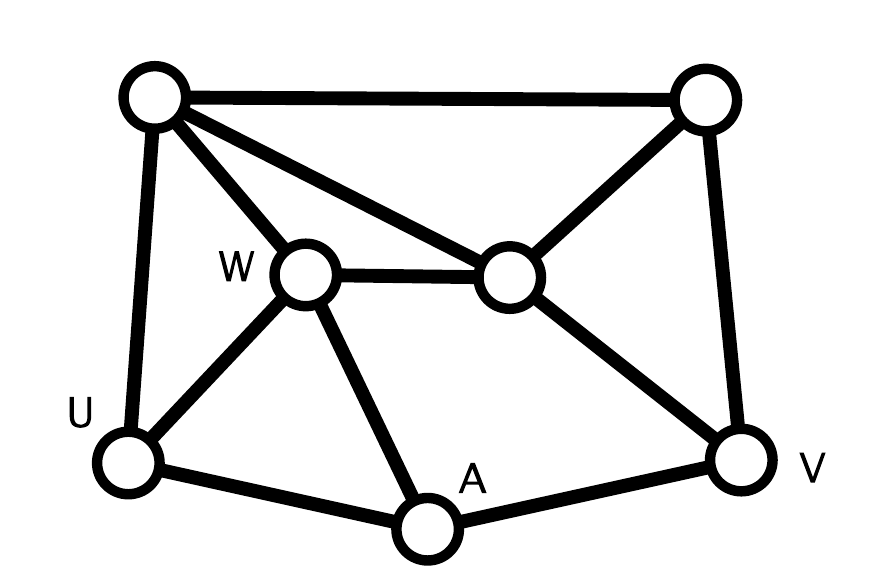}
	\caption{A Henneberg II extension of the Desargues-plus-one circuit.}
	\label{fig:hennebergII}
\end{figure}

\subparagraph{Inductive constructions for $3$-connected  circuits.} A {\em Henneberg II} extension (also called an {\em edge splitting} operation) is defined for an edge $uv$ and a non-incident vertex $w$, as follows: the edge $uv$ is removed, a new vertex $a$ and three new edges $au, av, aw$ are added. Berg and Jordan \cite{BergJordan} have shown that, if $G$ is a $3$-connected circuit, then a Henneberg II extension on $G$ is also a $3$-connected circuit. The {\em inverse Henneberg II} operation on a circuit removes one vertex of degree $3$ and adds a new edge among its three neighbors in such a way that the result is also a circuit. Berg and Jordan have shown that every $3$-connected circuit admits an inverse Henneberg II operation which also maintains $3$-connectivity. As a consequence, a $3$-connected circuit has an {\em inductive construction}, i.e. it can be obtained from $K_4$ by Henneberg II extensions that maintain $3$-connectivity. Their proof is based on the existence of two non-adjacent vertices with $3$-connected inverse Henneberg II circuits. We will make use in Section~\ref{sec:combRes} of the following weaker result, which does not require the maintenance of $3$-connectivity in the inverse Henneberg II operation.

\begin{lemma}[Theorem 3.8 in \cite{BergJordan}]
	\label{thm:BergJordan}Let $G=(V,E)$ be a $3$-connected circuit with $|V|\geq 5$. Then $G$ either has four vertices that admit an inverse Henneberg II that is a circuit, or $G$ has three pairwise non-adjacent vertices that admit an inverse Henneberg II that is a circuit (not necessarily $3$-connected).
\end{lemma}


\section{Combinatorial Resultants}
\label{sec:combRes}

We define now the {\em combinatorial resultant} operation on two graphs, prove Theorem~\ref{thm:combResConstruction} and describe its algorithmic implications. 

\begin{figure}[ht]
	\centering
	\includegraphics[width=0.24\textwidth]{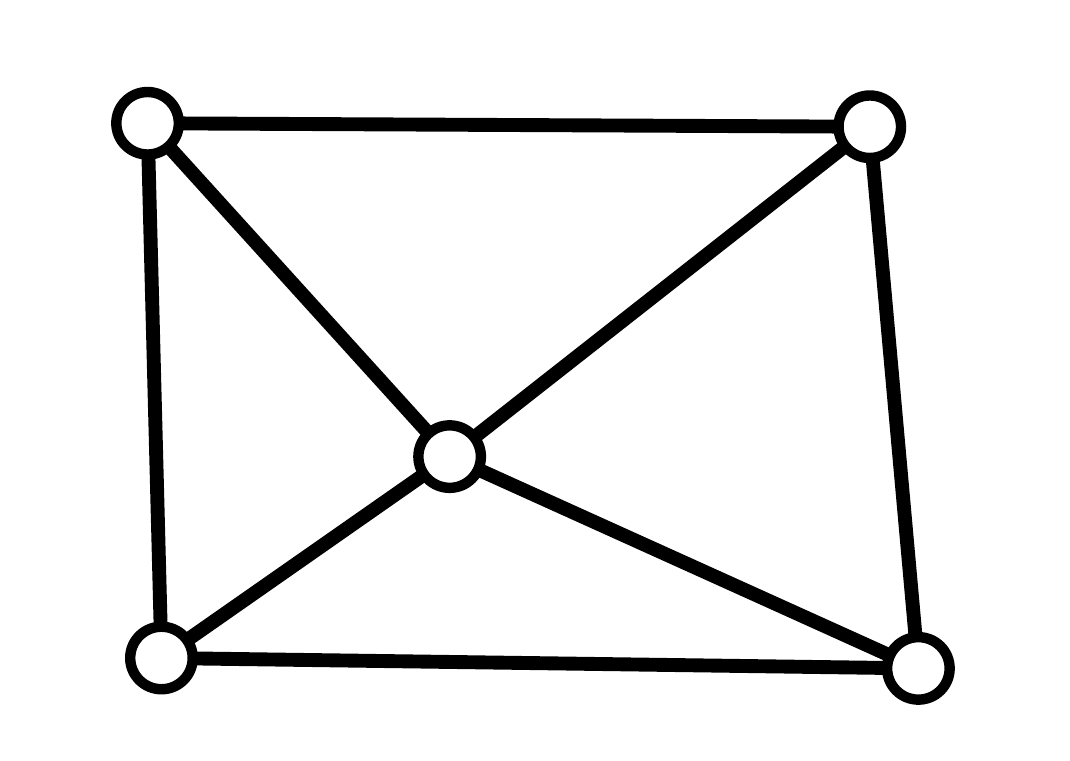}
	\includegraphics[width=0.24\textwidth]{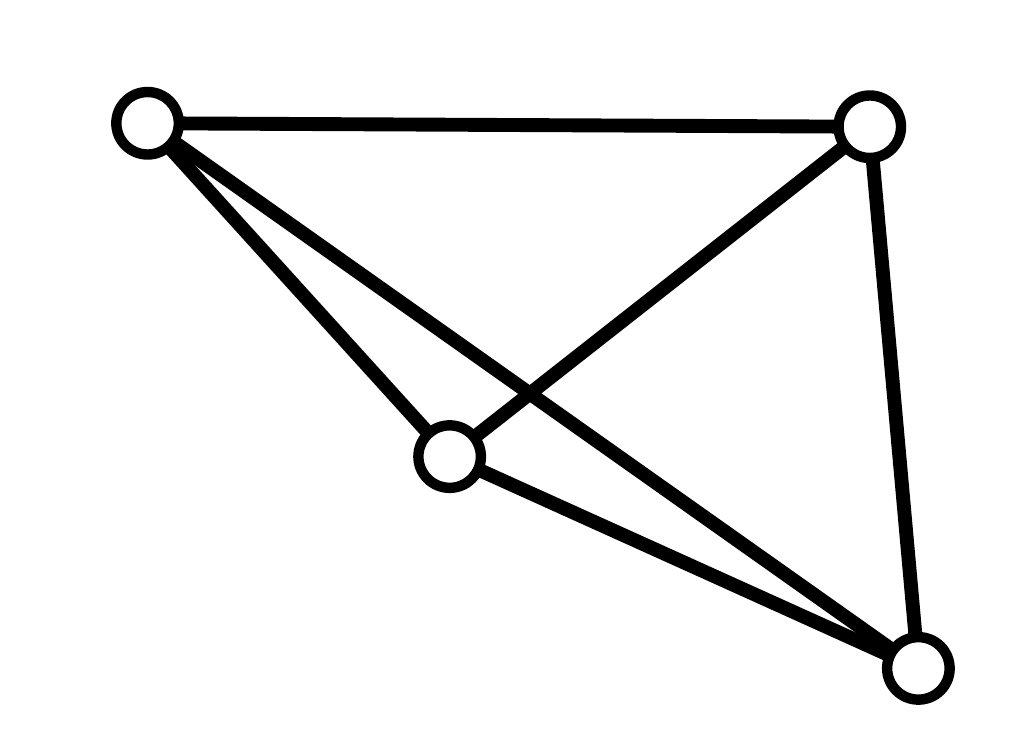}
	\includegraphics[width=0.24\textwidth]{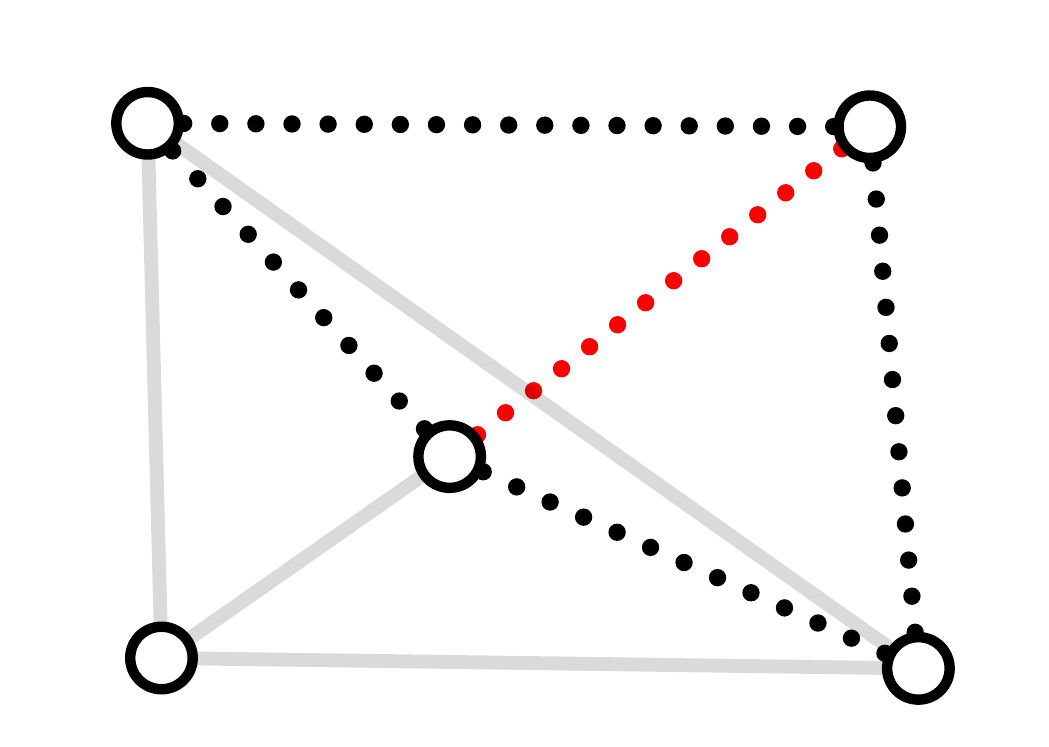}
	\includegraphics[width=0.24\textwidth]{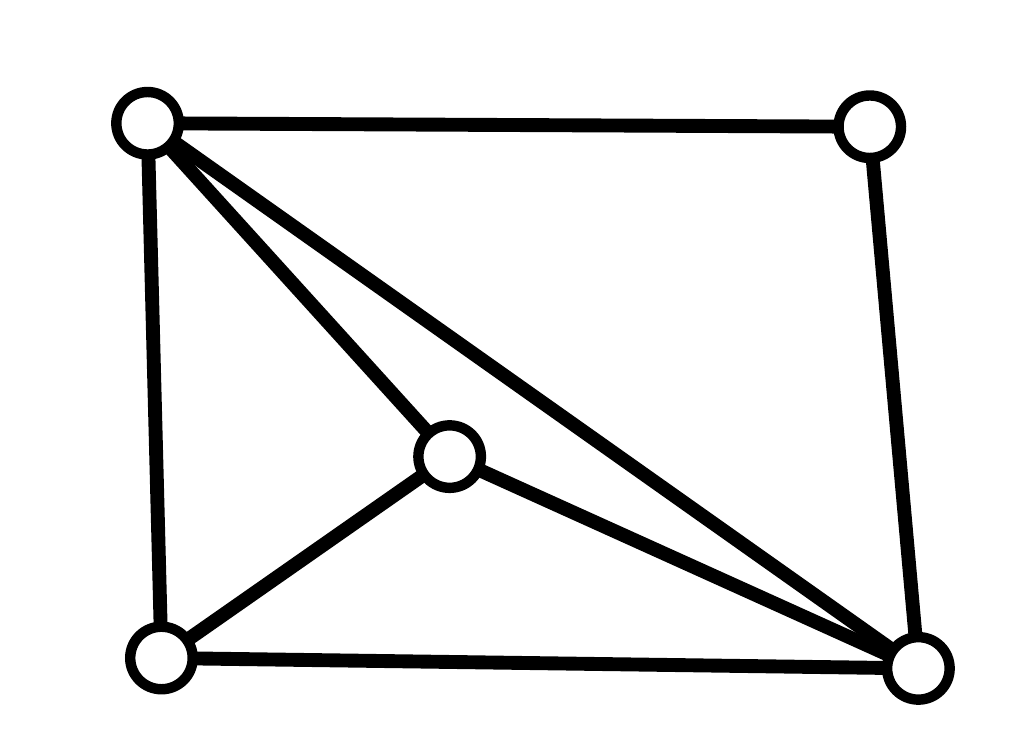}	
	\caption{A $4$-wheel $W_4$, a complete $K_4$ graph,  their common  Laman graph (dotted, with red elimination edge) and their combinatorial resultant, which is a Laman-plus-one graph but not a circuit.}
	\label{fig:graphW5K4}
\end{figure}

\subparagraph{Combinatorial resultant.}
	Let $G_1$ and $G_2$ be two distinct graphs with non-empty intersection $E_{\cap} \neq\emptyset$ and let $e\in E_{\cap}$ be a common edge. The {\em combinatorial resultant} of $G_1$ and $G_2$ on the {\em elimination edge} $e$ is the graph $\cres{G_1}{G_2}{e}$ with vertex set $V_{\cup}$ and edge set $E_{\cup}\setminus\{e\}$.

The 2-sum appears as a special case of a combinatorial resultant when the two graphs have exactly one edge in common, which is eliminated by the operation. Circuits are closed under the $2$-sum operation, but they are not closed under general combinatorial resultants (Fig.~\ref{fig:graphW5K4}). We are interested in combinatorial resultants that produce circuits from circuits.

\subparagraph{Circuit-valid combinatorial resultant sequences.} Two circuits are said to be properly intersecting if their common subgraph is Laman. 

\begin{lemma}
	\label{lem:combRes2n2}
	The combinatorial resultant of two circuits has $m=2n-2$ edges iff the common subgraph $G_{\cap}$ of the two circuits is Laman.
\end{lemma}

\begin{figure}[ht]
	\centering
	\includegraphics[width=0.24\textwidth]{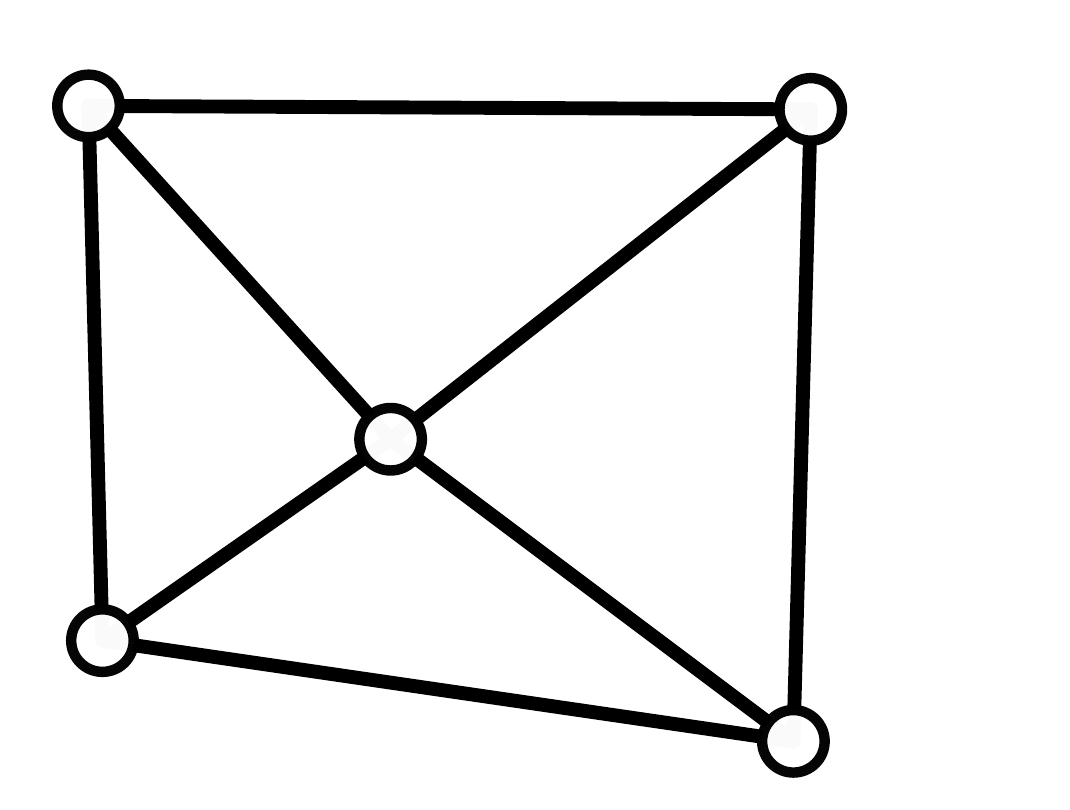}
	\hspace{-40pt}
	\includegraphics[width=0.24\textwidth]{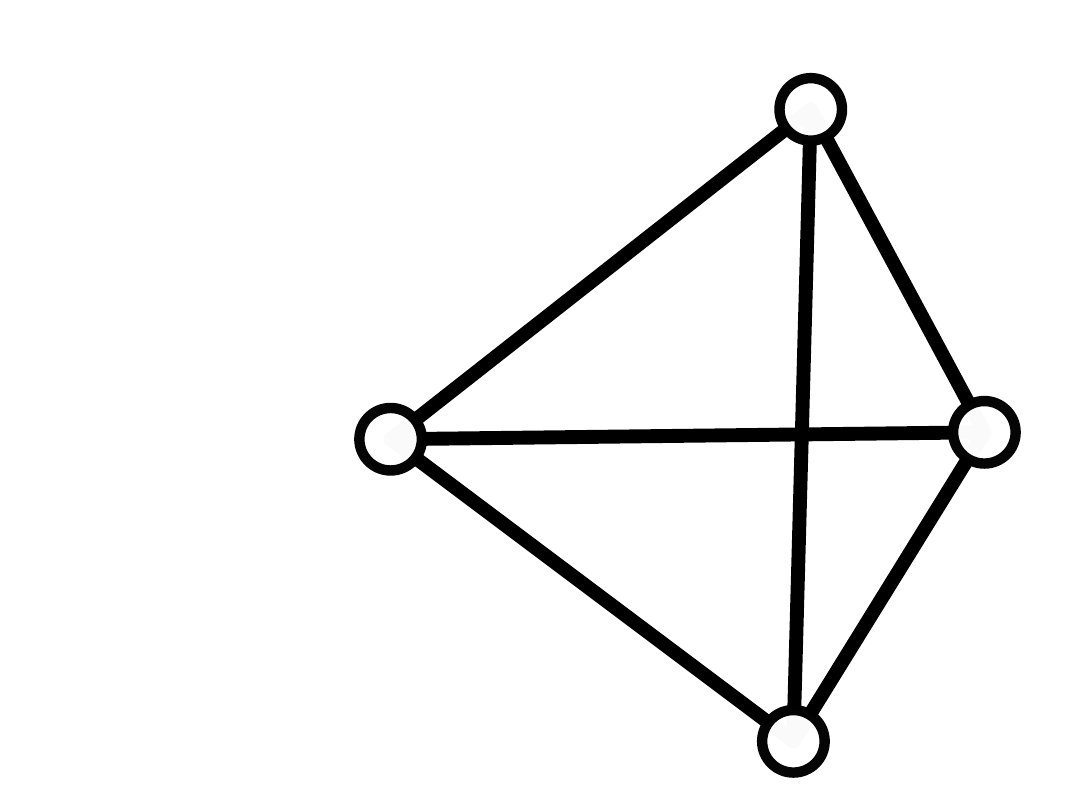}
	\includegraphics[width=0.24\textwidth]{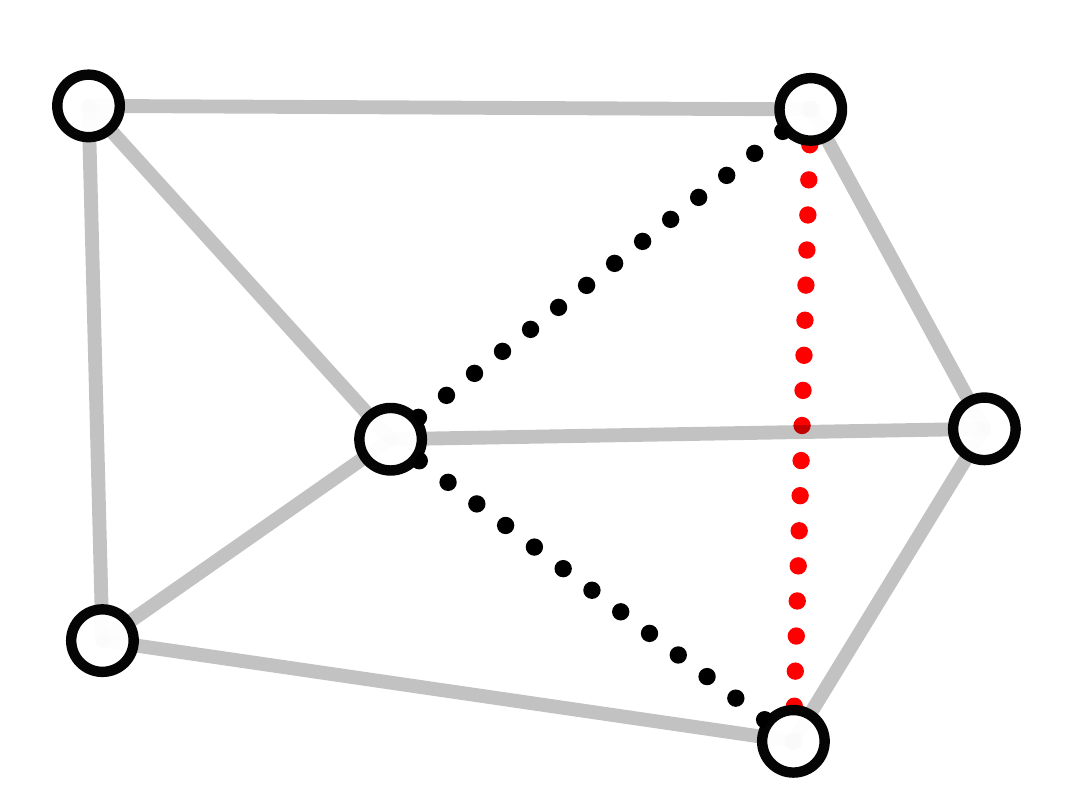}
	\includegraphics[width=0.24\textwidth]{w5.pdf}
	\caption{A $4$-wheel $W_4$ and a complete $K_4$ graph, their common  Laman graph (dotted, with red elimination edge) and their combinatorial resultant, the $5$-wheel $W_5$ circuit.}
	\label{fig:combResultant}
\end{figure}

\begin{proof} 
Let $C_1$ and $C_2$ be two circuits with $n_i$ vertices and $m_i$ edges, $i=1,2$, and let $C$ be their combinatorial resultant with $n$ vertices and $m$ edges.
By inclusion-exclusion $n=n_1+n_2-n_{\cap}$ and $ m = m_1 + m_2 - m_{\cap} - 1$.
Substituting here the values for $m_1=2n_1-2$ and $m_2=2n_2-2$, we get $ m = 2 n_1 - 2 + 2 n_2 - 2 - m_{\cap} - 1 = 2(n_1 + n_2 - n_{\cap}) -2 + 2 n_{\cap} - 3 - m_{\cap} = (2n-2) + (2 n_{\cap} - 3) - m_{\cap}$. We have $m = 2 n -2$ iff $m_{\cap} = 2 n_{\cap} - 3$. Since both $C_1$ and $C_2$ are circuits, it is not possible that one edge set be included in the other: circuits are minimally dependent sets of edges and thus cannot contain other circuits. As a proper subset of both $E_1=E(C_1)$ and $E_2=E(C_2)$, $E_{\cap}$  satisfies the hereditary $(2,3)$-sparsity property. If furthermore $G_{\cap}$ has exactly $2 n_{\cap} - 3$ edges, then it is Laman.
\end{proof}

\noindent
A combinatorial resultant operation applied to two properly intersecting circuits is said to be {\em circuit-valid} if it results in a spanning circuit. An example is shown in Fig.~\ref{fig:combResultant}. Being properly intersecting is a necessary but not sufficient condition for the combinatorial resultant of two circuits to produce a circuit (Fig.~\ref{fig:graphW5K4}).

\begin{problem}
Find necessary and sufficient conditions for the combinatorial resultant of two circuits to be a circuit.
\end{problem}

In Section \ref{sec:prelimRigidity} we have seen that a $2$-connected circuit can be obtained from $3$-connected circuits via $2$-sums. The proof of Theorem \ref{thm:combResConstruction} is completed by Proposition~\ref{prop:circuits3conn} below.

\begin{proposition}
	\label{prop:circuits3conn}
	Let $C=(V,E)$ be a $3$-connected circuit spanning $n+1\geq 5$ vertices. Then, in polynomial time, we can find two circuits $A$ and $B$ such that $A$ has $n$ vertices, $B$ has at most $n$ vertices and $C$ can be represented as the combinatorial resultant of $A$ and $B$.
\end{proposition}

\begin{proof}
We apply a weaker version of Lemma~\ref{thm:BergJordan} to find two non-adjacent vertices $a$ and $b$ of degree 3 such that a circuit $A=(V\setminus{\{a\}},E\setminus{\{au,av,aw\}})$ can be produced via an inverse Henneberg II operation on vertex $a$ in $C$ (see Fig.~\ref{fig:invCombResA}). Let the neighbors of vertex $a$ be $N(a)=\{u,v,w\}$ such that $e=uw$ was not an edge of $C$ and is the one added to obtain the new circuit $A$. 
\begin{figure}[ht]
	\centering
		\includegraphics[width=.33\textwidth]{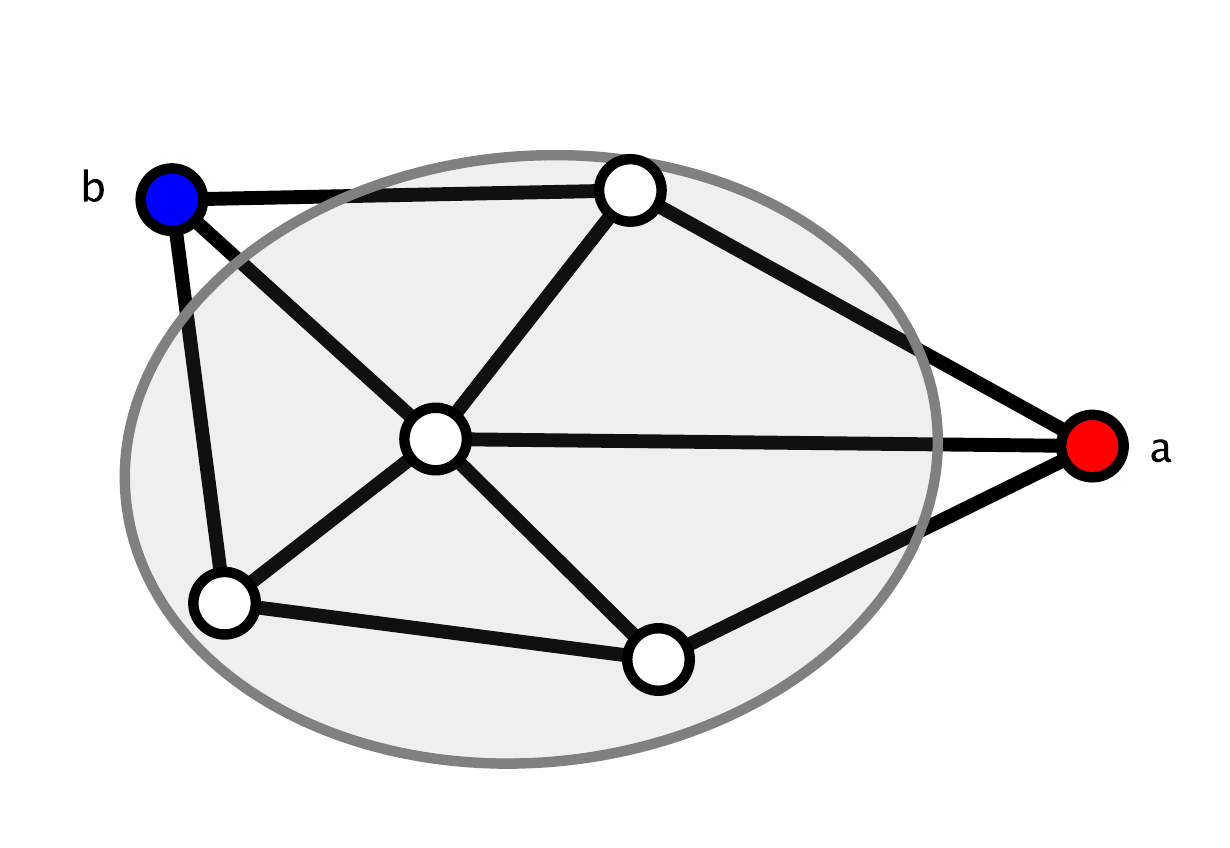}
		\includegraphics[width=.3\textwidth]{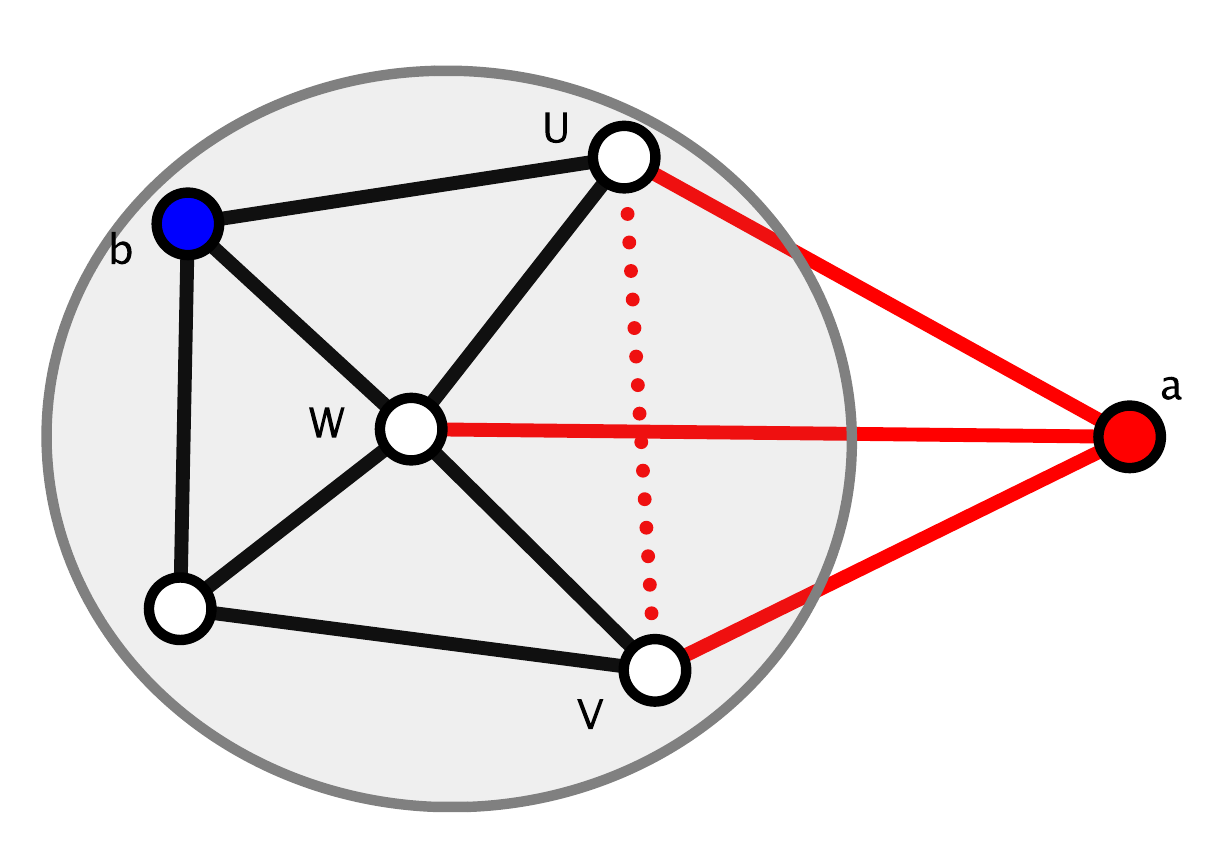}
		\includegraphics[width=.3\textwidth]{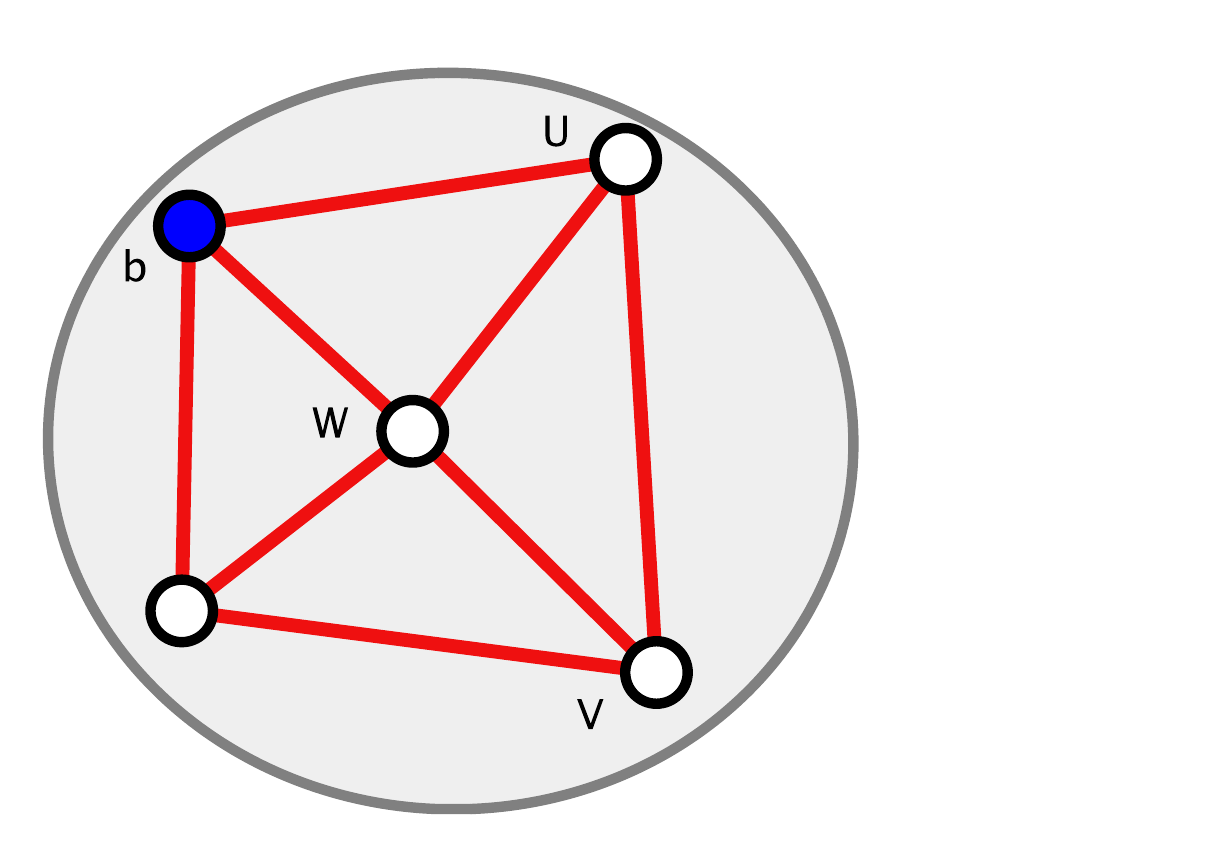}
	\caption{
	The 3-connected circuit $C$ spanning $n+1$ vertices with two non-adjacent vertices $a$ (red) and $b$ (blue) of degree 3. Note that $N(a)$ and $N(b)$ may not be disjoint. An inverse Henneberg II at $a$ removes the red edges at $a$ and adds dotted red edge $e = uv$. Circuit $A$ (red). 
	}
	\label{fig:invCombResA}
\end{figure}
To define circuit $B$, we first let $L$ be the subgraph of $C$ induced by $V\setminus\{b\}$. Simple sparsity consideration show that $L$ is a Laman graph. The graph $D$ obtained from $L$ by adding the edge $e=uv$, as in Fig.~\ref{fig:invCombResB} (left), is a Laman-plus-one graph 
containing the three edges incident to $a$ (which are not in $A$) and the edge $e$ (which is in $A$). $D$ contains a unique circuit $B$ (Fig.~\ref{fig:invCombResB} left) with edge $e\in B$ (see e.g.~\cite[Proposition 1.1.6]{Oxley:2011}). It remains to prove that $B$ contains $a$ and its three incident edges. If $B$ does not contain $a$, then it is a proper subgraph of $A$. But this contradicts the minimality of $A$ as a circuit. Therefore $a$ is a vertex in $B$, and because a vertex in a circuit can not have degree less than $3$, $B$ contains all its three incident edges.

\begin{figure}[ht]
	\centering
		\includegraphics[width=.33\textwidth]{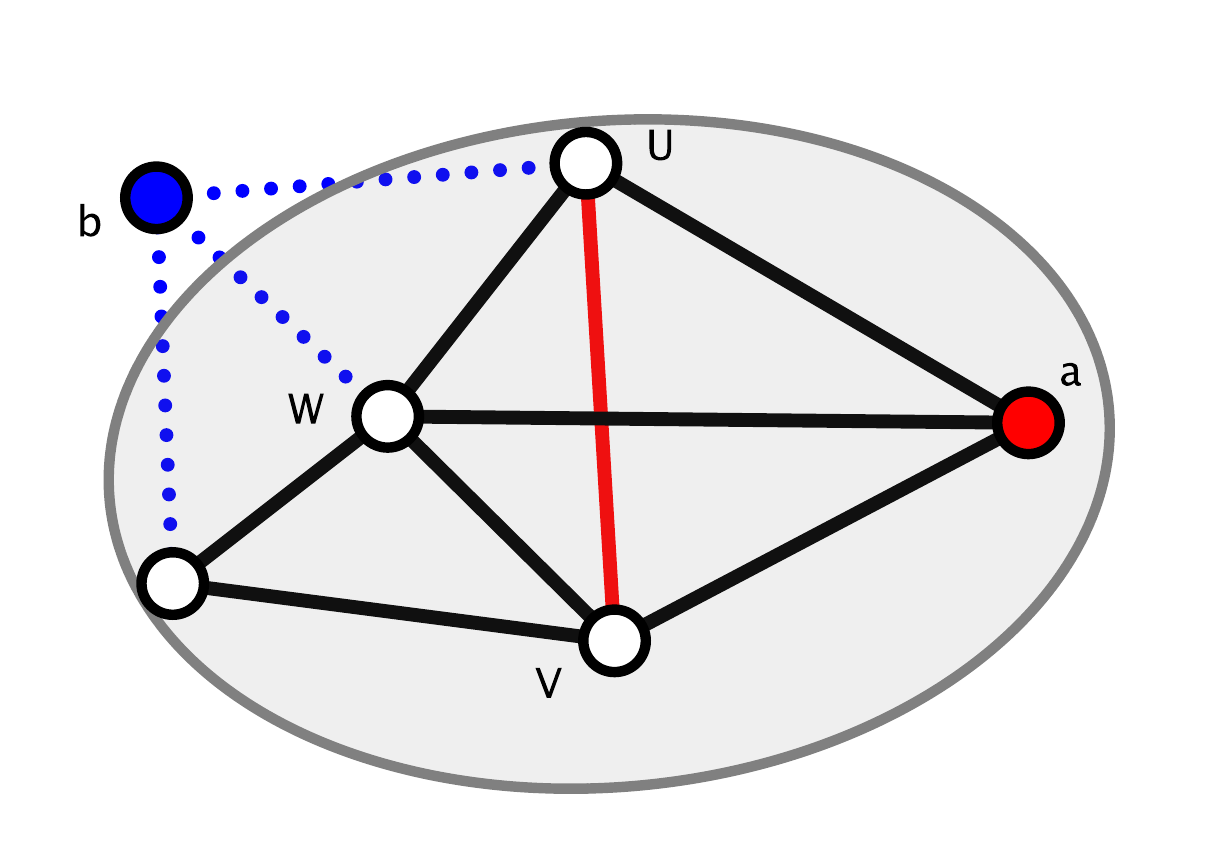}
		\includegraphics[width=.3\textwidth]{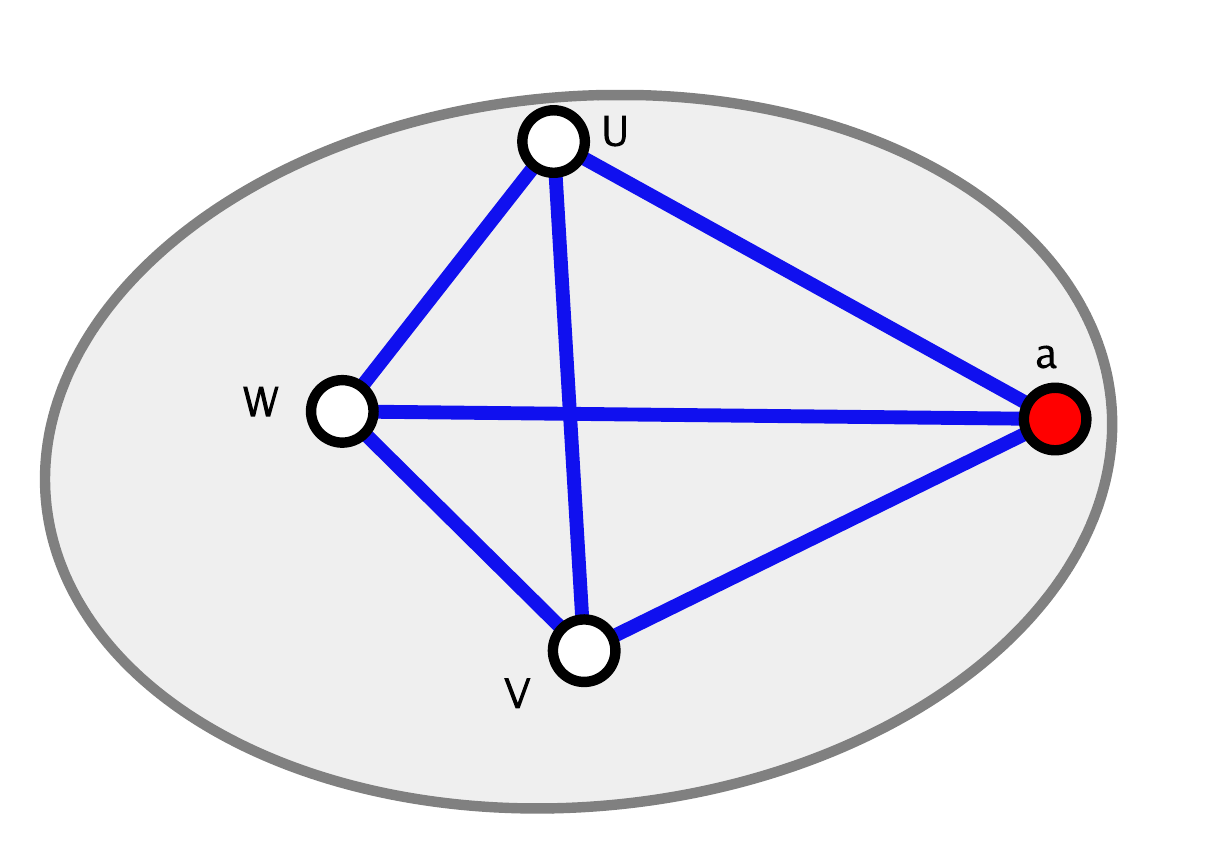}
	\caption{
	Remove from $C$ the edges from $b$ (blue dotted) and add red edge $e$. Circuit $B$ (blue).}
	\label{fig:invCombResB}
\end{figure}

The combinatorial resultant $\cres{A}{B}{e}$ of the circuits $A$ and $B$ with $e$ the eliminated edge satisfies the desired property that $C=\cres{A}{B}{e}$.

The algorithm below 
captures this procedure. The main steps, the Inverse Henneberg II step on a circuit at line \ref{line:invH} and finding the unique circuit in a Laman-plus-one graph at line \ref{line:circuit} can be done in polynomial time using properties of the $(2,3)$ and $(2,2)$-sparsity pebble games from \cite{streinu:lee:pebbleGames:2008}. 
\end{proof}

\begin{algorithm}[ht]
	\caption{Inverse Combinatorial Resultant}
	\label{alg:comb}
	\textbf{Input}: $3$-connected circuit $C$\\
	\textbf{Output}: circuits $A$, $B$ and edge $e$ such that $C=\cres{A}{B}{e}$
	\begin{algorithmic}[1]
		\For{each vertex $a$ of degree $3$}
		\If{inverse Henneberg II is possible on $a$\\ \ \ \ \ \textbf{and} there is a non-adjacent degree $3$  vertex $b$}
		\State Get circuit $A$ and edge $e$ by inverse Henneberg II in $C$ on $a$\label{line:invH}
		\State Let $D = C$ without $b$ (and its edges) and with new edge $e$
		\State Compute unique circuit $B$ in $D$ \label{line:circuit}
		\State \Return circuits $A, B$ and edge $e$
		\EndIf
		\EndFor
	\end{algorithmic}
\end{algorithm}

\subparagraph{Resultant tree.} The inductive construction of a circuit using combinatorial resultant operations can be represented in a {\em tree} structure. Let $C$ be a rigidity circuit with $n$ vertices. A \emph{resultant tree} $T_C$ for the circuit $C$ is a rooted binary tree with $C$ as its root and such that: (a) the nodes of $T_C$ are circuits; (b) circuits on level $l$ have at most $n-l$ vertices; (c) the two children $\{C_j,C_k\}$ of a parent circuit $C_i$ are such that $C_i=\cres{C_j}{C_k}{e}$, for some common edge $e$, and (d) the leaves are complete graphs on 4 vertices. The complexity of finding a resultant tree depends on the size of the tree, whose depth is at most $n-4$. The combinatorial resultant tree may thus be anywhere between linear to exponential in size. The best case occurs when the resultant tree is path-like, with each internal node having a $K_4$ leaf. The worst case could be a complete binary tree: each internal node at level $k$ would combine two circuits with the same number of vertices $n-k-1$ into a circuit with $n-k$ vertices. Sporadic examples of small balanced combinatorial resultant trees exist (e.g. $K_{33}$-plus-one), but it remains an open problem to find infinite families of such examples. Even if such a family would be found, it is still conceivable that alternative, non-balanced combinatorial resultant trees could yield the same circuit. 


\begin{problem}
	Characterize the circuits produced by the worst-case size of the combinatorial resultant tree. 
\end{problem}

\begin{problem}
	Are there infinite families of circuits with only balanced combinatorial resultant trees? 
\end{problem}

\begin{problem}
	Refine the time complexity analysis of the combinatorial resultant tree algorithm.
\end{problem}

\begin{corollary}\label{cor:nonunique}
The representation of $C$ as the combinatorial resultant of two smaller circuits is not unique, in general. An example is the ``double-banana'' 2-connected circuit shown in Figure \ref{fig:exampleCircuit}.
\end{corollary}

\begin{figure}[ht]
\centering
\includegraphics[width=.3\textwidth]{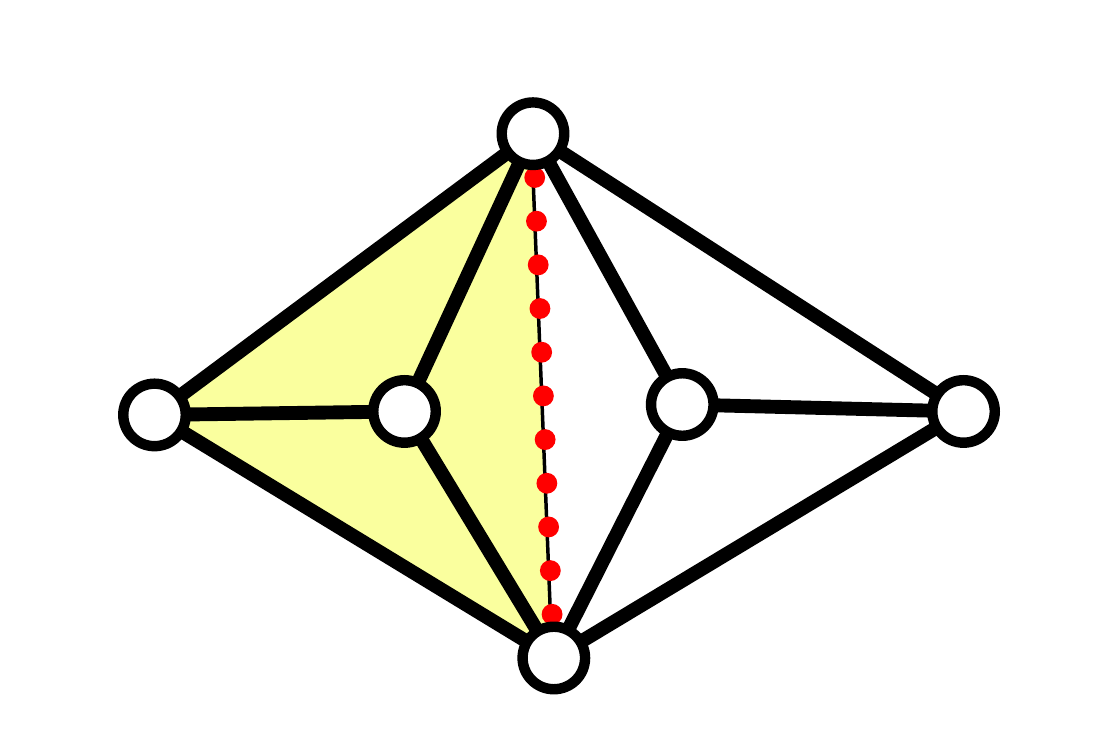}
\includegraphics[width=.3\textwidth]{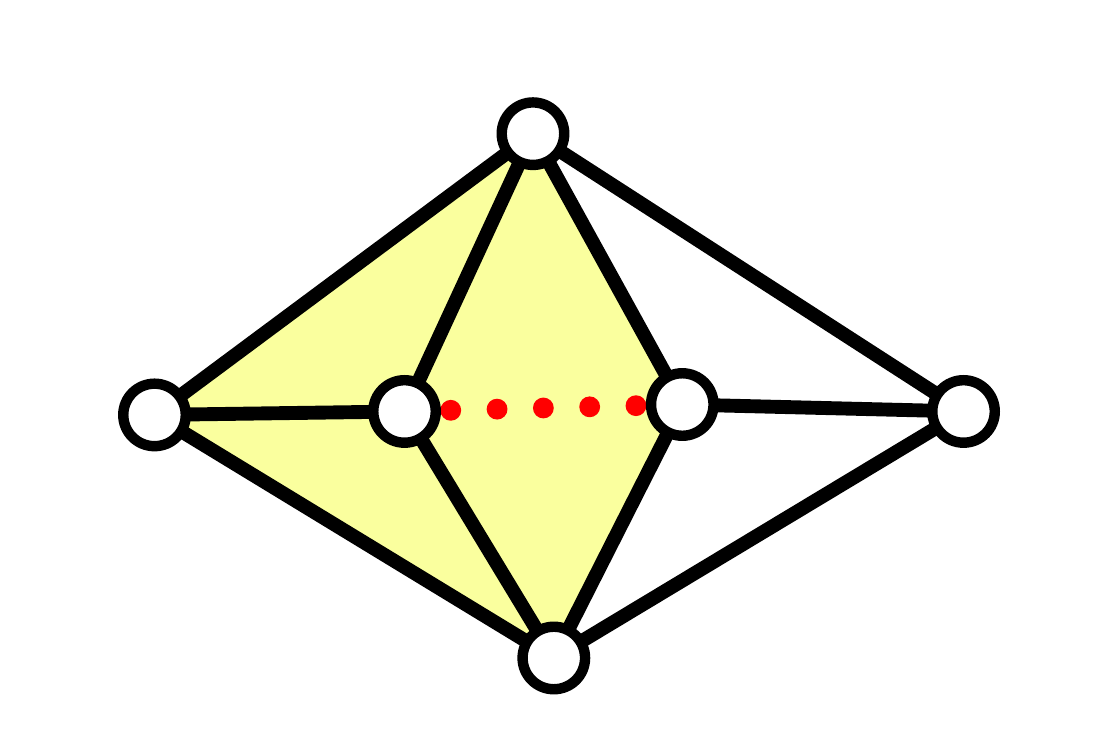}
\caption{The 2-connected {\em double-banana} circuit can be obtained as combinatorial resultant from two $K_4$ graphs (left, $2$-sum), and from two wheels on 4 vertices (right). Dashed lines indicate the eliminated edges, and in each case one of the two circuits is highlighted to distinguish $K_4$ from $W_4$.}
\label{fig:exampleCircuit}
\end{figure}

\section{Preliminaries: Resultants and Elimination Ideals}
\label{sec:prelimResultants}

We turn now to the algebraic aspects of our problem. In this section we review known concepts and facts about resultants and elimination ideals that are essential ingredients of our proofs in Section~\ref{sec:algResCircuits}.

\subparagraph{Resultants.} The resultant can be introduced in several equivalent ways \cite{GelfandKapranovZelevinsky}. Here we use its definition as the determinant of the Sylvester matrix.

Let $R$ be a ring of polynomials and $f,g\in R[x]$ with $\deg_xf=r$ and $\deg_xg=s$ be such that at least one of $r$ or $s$ is positive and
\begin{align*}
 f(x)&=a_{r}x^{r}+\cdots+a_1 x + a_0,\\
 g(x)&=b_{s}x^{s}+\cdots+b_1 x + b_0 
\end{align*}

The \emph{resultant} of $f$ and $g$ with respect to the indeterminate $x$, denoted $\res{f}{g}{x}$, is the determinant of the $(r+s)\times(r+s)$ Sylvester matrix
\[
\operatorname{Syl}(f,g,x)=\begin{pmatrix}
a_{r} & a_{r-1} & a_{r-2} & \cdots & a_0 & 0 & 0 & \cdots & 0\\
0 & a_r & a_{r-1} & \cdots & a_1 & a_0 & 0 & \cdots & 0\\
0 & 0 & a_r & \cdots & a_2 & a_1 & a_0 & \cdots & 0\\
\vdots & \vdots & \vdots & \ddots & \vdots & \vdots & \vdots & \ddots & 0\\
0 & 0 & 0 & \cdots & a_r & a_{r-1} & a_{r-2} & \ddots & a_0\\

b_{s} & b_{s-1} & b_{s-2} & \cdots & b_0 & 0 & 0 & \cdots & 0\\
0 & b_s & b_{s-1} & \cdots & b_1 & b_0 & 0 & \cdots & 0\\
0 & 0 & b_s & \cdots & b_2 & b_1 & b_0 & \cdots & 0\\
\vdots & \vdots & \vdots & \ddots & \vdots & \vdots & \vdots & \ddots & 0\\
0 & 0 & 0 & \cdots & b_s & b_{s-1} & b_{s-2} & \ddots & b_0
\end{pmatrix}
\]
where the submatrix $S_f$ containing only the coefficients of $f$ is of dimension $s\times (r+s)$, and the submatrix $S_g$ containing only the coefficients of $g$ is of dimension $r\times (r+s)$. Unless $r=s$, the columns $(a_{0}~a_{1}~\cdots~a_r)$ and $(b_{0}~b_{1}~\cdots~b_s)$ of $S_f$ and $S_g$, respectively, are not aligned in the same column of $\operatorname{Syl}(f,g,x)$, as displayed above, but rather the first is shifted to the left or right of the second, depending on the relationship between $r$ and $s$. We will make implicit use of the following well-known symmetric and multiplicative properties of the resultant:

\begin{proposition}
	\label{prop:basicPropResultants}(\cite[pp.~398]{GelfandKapranovZelevinsky}) Let $f, g, h\in R[x]$. The resultant of $f$ and $g$ satisfies
\begin{itemize}
	\item $\res{f}{g}{x}=(-1)^{rs}\res{g}{f}{x}$,
	\item $\res{fg}{h}{x}=\res{f}{h}{x}\res{g}{h}{x}$.
\end{itemize}
\end{proposition}

\begin{proposition}\label{prop:resultantCommonFactor}Let $R$ be a unique factorization domain and $f,g\in R[x]$. Then $f$ and $g$ have a common factor in $R[x]$ if and only if $\res{f}{g}{x}=0$. 
\end{proposition}

This proposition is stated in \cite[pp.~9]{GriffithsHarris} without proof. When $R$ is a field, a proof of this property can be found in \cite[Chapter 3, Proposition 3 of \S6]{CoxLittleOshea}, which directly generalizes to polynomial rings via Hilbert's Nullstellensatz.

\subparagraph{Homogeneous properties.} Circuit polynomials in the Cayley-Menger ideal are homogeneous polynomials (see Proposition \ref{prop:circHomogeneous}), hence we are interested in the properties of the resultant of homogenous polynomials. It is well known that the resultant $\res{f}{g}{x}$ of two homogeneous polynomials is itself homogeneous, and with degree $m\deg_x{g}+n\deg_{x}{f}-\deg_x{f}\cdot\deg_x{g}$, where $m$, resp.\ $n$ are the homogeneous degrees of $f$, resp.\ $g$. For completeness we prove these two facts; the exposition follows \cite{CoxLittleOshea}.

Let $m$, $n$, $r$ and $s$ be positive integers such that $m\geq r$ and $n\geq s$. Let $f$ and $g$ be polynomials of degree $r$ and $s$ in $x$, with generic coefficients $a_{m-r}$, $\dots$, $a_m$ and $b_{n-s}$, $\dots$, $b_n$, respectively, i.e.
\begin{align*}
 f(x)&=a_{m-r}x^{r}+\cdots+a_{m-1} x + a_m,\\
 g(x)&=b_{n-s}x^{s}+\cdots+b_{n-1} x + b_n. 
\end{align*}

\begin{proposition}
The resultant $\res{f}{g}{x}$ is a homogeneous polynomial in the ring $\mathbb Z[a_{m-r},\dots,a_m,$ $b_{n-s}, \dots, b_n]$ of degree $r+s$.\end{proposition}

\begin{proof}This property follows from the Leibniz expansion of the determinant of the Sylvester matrix $\operatorname{Syl}(f,g,x)=(S_{i,j})$, where $S_{i,j}$ is the $(i,j)$-th entry, which shows that each term of $\res{f}{g}{x}$ is, up to sign, equal to
\[\prod_{i=1}^{r+s}{S_{i,\sigma(i)}}\]
for some permutation $\sigma$ of the set $[r+s]$. This term is non-zero if and only if $S_{i,\sigma(i)}\neq 0$ for all $i\in[r+s]$, and since $S_{i,\sigma(i)}\in\{a_{m-r},\dots, a_m, b_{n-s}, \dots, b_n\}$, the homogeneous degree of $\res{f}{g}{x}$ is $r+s$.\end{proof}

\begin{proposition}
	\label{prop:resultantHomogeneous}
	Let $f=a_{m-r}x^{r}+\cdots+a_{m-1} x + a_m$ and $g=b_{n-s}x^{s}+\cdots+b_{n-1} x + b_n$ be homogeneous polynomials in $k[y_1,\dots, y_t,x]$ of homogeneous degree $m$ and $n$, respectively, so that $a_i,b_j\in k[y_1,\dots,y_t]$ are homogeneous of degree $i$, for all $i\in\{m-r,\dots,m\}$ and all $j\in\{n-s,\dots,n\}$. If $\res{f}{g}{x}\neq 0$, then it is a homogeneous polynomial in $k[y_1,\dots,y_t]$ of degree
\[m\deg_x{g}+n\deg_{x}{f}-\deg_x{f}\cdot\deg_x{g}=ms + nr-rs.\]
\end{proposition}

We were not able to find a reference for this proposition in the literature, however the case when $f$, resp.\ $g$ are homogeneous of degree $r$, resp.\ $s$ and such that $\deg_xf=r$ and $\deg_xg=s$, so that
\begin{align*}
 f&=a_{0}x^{r}+\cdots+a_{1} x + a_r,\\
 g&=b_{0}x^{s}+\cdots+b_{1} x + b_s, 
\end{align*}
can be found in e.g.\ \cite[pp.~454]{CoxLittleOshea} as Lemma 5 of $\S7$ of Chapter 8, stating that in that case $\res{f}{g}{x}$ is of homogeneous degree $rs$. The proof below is a direct adaptation of the proof of Lemma 5 in \cite[pp.~454]{CoxLittleOshea}, and the Lemma 5 itself follows directly from Proposition \ref{prop:resultantHomogeneous} by substituting $m\to r$ and $n\to s$ so to obtain $rs+sr-rs=rs$.

\begin{proof} Let $\operatorname{Syl}(f,g,x)=(S_{i,j})$ be the Sylvester matrix of $f$ and $g$ with respect to $x$, and let, up to sign, $\prod_{i=1}^{r+s}S_{i,\sigma(i)}$ be a non-zero term in the Leibniz expansion of its determinant for some permutation $\sigma$ of $[r+s]$.

A non-zero entry $S_{i,\sigma(i)}$ has degree $m-(r+i-\sigma(i))$ if $1\leq i\leq s$ and degree $n-(i-\sigma(i))$ if $s+1\leq i\leq r+s$. Therefore, the total degree of $\prod_{i=1}^{r+s}S_{i,\sigma(i)}$ is
\begin{align*}
&\sum_{i=1}^s[m-(r+i-\sigma(i))]+\sum_{i=s+1}^{r+s}[n-(i-\sigma(i))]
=\sum_{i=1}^s(m-r)+\sum_{i=s+1}^{s+r}n-\sum_{i=1}^{r+s}(i-\sigma(i))\\
=&s(m-r)+rn-0=m\deg_xg+n\deg_xf-\deg_xf\cdot\deg_xg.\end{align*}\end{proof}

\subparagraph{Elimination ideals.}

Let $I$ be an ideal of $\mathbb Q[X]$ and $X'\subset X$ non-empty. The \emph{elimination ideal} of $I$ with respect to $X'$ is the ideal $I\cap \mathbb Q[X']$ of the ring $\mathbb Q[X']$.

Elimination ideals frequently appear in the context of \grobner{} bases \cite{Buchberger, CoxLittleOshea} which give a general approach for computing elimination ideals: if $\mathcal G$ is a \grobner{} basis for $I$ with respect to an \emph{elimination order} (see Exercises 5 and 6 in \S1 of Chapter 3 in \cite{CoxLittleOshea}), e.g.\ the lexicographic order $x_{i_1}>x_{i_2}>\dots>x_{i_n}$, then the elimination ideal $I\cap \q[x_{i_{k+1}},\dots,x_{i_n}]$ which eliminates the first $k$ indeterminates from $I$ in the specified order has $\mathcal G\cap \q[x_{i_{k+1}},\dots,x_{i_n}]$ as its \grobner{} basis.

We will frequently make use of the following well-known result.

\begin{proposition} If $I$ is a prime ideal of $\q[X]$ and $X'\subset X$ is non-empty,
then the elimination ideal $I\cap \q[X']$ is prime.
\end{proposition}

\begin{proof}
	If $f\cdot g\in I\cap\q[X']$ then certainly $f\cdot g\in I$ so at least one of $f$ or $g$ is in $I\cap\q[X']$.
\end{proof}

Let $X'\subset X$ be non-empty and $R=\mathbb Q[X']$. Furthermore, let $f,g\in R[x]$, where $x\in X\setminus X'$. It is clear from the definition of the resultant that $\res{f}{g}{x}\in R$. In Section \ref{sec:algResCircuits} we will make use of the following proposition.

\begin{proposition}\label{prop:resultantElimination}
Let $I$ be an ideal of $R[x]$ and $f,g\in I$. Then $\res{f}{g}{x}$ is in the elimination ideal $I\cap R$.
\end{proposition}

A proof of this proposition can be found in \cite[pp.~167]{CoxLittleOshea}.

%

\section{Preliminaries: Ideals and Algebraic Matroids}
\label{sec:prelimAlgMatroids}

Recall that a set of vectors in a vector space is linearly dependent if there is a non-trivial linear relationship between them. Similarly, given a finite collection $A$ of complex numbers, we say that $A$ is \emph{algebraically dependent} if there is a non-trivial polynomial relationship between the numbers in $A$.

More precisely, let $k$ be a field (e.g.\ $k=\q$) and $k\subset F$ a field extension of $k$. Let $A=\{\alpha_1,\dots,\alpha_n\}$ be a finite subset of $F$.

\begin{definition}We say that $A$ is algebraically dependent over $k$ if there is a non-zero (multivariate) polynomial with coefficients in $k$ vanishing on $A$. Otherwise, we say that $A$ is algebraically independent over $k$.\end{definition}

\subparagraph{Algebraic independence and algebraic matroids.}
\label{app:subsection:AlgDepMat} 
It was noticed by van der Waerden that the algebraically independent subsets $A$ of a finite subset $E$ of $F$ satisfy matroid axioms \cite{VDWmoderne,VDW2} and therefore define a matroid called \emph{the algebraic matroid on $E$ over $k$}.

\begin{definition}\label{def:algmat}Let $k$ be a field and $k\subset F$ a field extension of $k$. Let $E=\{\alpha_1,\dots,\alpha_n\}$ be a finite subset of $F$. The \emph{algebraic matroid on $E$ over $k$} is the matroid $(E,\mathcal I)$ such that $I\in\mathcal I$ if and only if $I$ is algebraically independent over $k$.
\end{definition}

In this paper we use an equivalent definition of algebraic matroids in terms of polynomial ideals. Before stating this equivalent defintion, we will first recall some elementary definitions and properties of ideals in polynomial rings. For a general reference on polynomial rings the reader may consult \cite{Lang}.

\subparagraph{Notations and conventions.} To keep the presentation focused on the goal of the paper, we refrain from giving the most general form of a statement or a proof. We work over the field of rational numbers $\q$. In this section, the set of variables $X_n$ denotes $X_n = \{x_i: 1\leq i\leq n\}$. Polynomial rings $R$ are always of the form $R=\q[X]$, over sets of variables $X\subset X_n$.  The \emph{support} $\supp f$ of a polynomial $f\in \q[X_n]$ is the set of indeterminates appearing in $f$. The degree of a variable $x$ in a polynomial $f$ is denoted by $\deg_xf$.

\subparagraph{Polynomial ideals.} A set of polynomials $I \subset\q[X]$ is an \emph{ideal of} $\q[X]$ if it is closed under addition and multiplication by elements of $\q[X]$. Every ideal contains the zero ideal $\{0\}$, and if an ideal $I$ contains an element of $\q$, then $I=\q[X]$.
A {\em generating set} for an ideal is a set $S\subset \q[X]$ of polynomials such that every polynomial in the ideal is an algebraic combination (addition and multiplication) of elements in $S$ with coefficients in $\q[X]$. {\em Hilbert Basis Theorem} (see \cite{CoxLittleOshea}) guarantees that every ideal in a polynomial ring has a finite generating set. Ideals generated by a single polynomial are called \emph{principal}. An ideal $I$ is a \emph{prime} ideal if, whenever $fg\in I$, then either $f\in I$ or $g\in I$. A polynomial is {\em irreducible} (over $\q$) if it cannot decomposed into a product of non-constant polynomials in $\q[X]$. A principal ideal is prime iff it is generated by an irreducible polynomial. However, an ideal generated by two or more irreducible polynomials is not necessarily prime. 

Let $I$ be an ideal of $R=\q[X_n]$. A \emph{minimal prime ideal over} $I$ is a prime ideal of $R$ minimal among all prime ideals containing $I$ with respect to set inclusion. By Zorn's lemma a proper ideal of $\q[X_n]$ always has at least one minimal prime ideal above it.


\subparagraph{Dimension.} 
By definition, the \emph{dimension of the ring} $\q[X_n]$ is $n$. This definition of dimension of a polynomial ring is a special case of the more general concept of \textit{Krull dimension} of a commutative ring. The \emph{dimension} $\dim I$ of an ideal $I$ of $\q[X_n]$ is the cardinality of the maximal subset $X\subseteq X_n$ with the property $I\cap \q[X]=\{0\}$.

We say that a (strict) chain $I_0\subset I_1 \subset \cdots \subset I_h$ of prime ideals of $\q[X_n]$ has length $h$. The \emph{codimension} or \emph{height} $\codim I$ of a prime ideal $I$ of $\q[X_n]$ is defined as the supremum of the lengths of maximal chains of prime ideals $\{0\}=I_0\subset I_1 \subset \cdots \subset I_h$ such that $I_h=I$.

In a polynomial ring $\q[X_n]$ all prime ideals have a finite height and any two maximal chains of prime ideals terminating at $I$ have the same height. Furthermore, we have
$\dim I + \codim I = n$. 

An important bound on the codimension of a prime ideal $I\subset\q[X_n]$ is given by: 
\begin{theorem}[Krull's Height Theorem] Let $J\subset\q[X_n]$ be an ideal generated by $m$ elements, and let $I$ be a prime ideal over $J$. Then $\codim I\leq m$. Conversely, if $I$ is a prime ideal such that $\codim I\leq m$, then it is a minimal prime of an ideal generated by $m$ elements.
\end{theorem}

Krull's Height Theorem holds more generally for all Noetherian rings, see \cite[Chapter 10]{Eisenbud} or \cite[Chapter 7]{Jacobson2}. An immediate consequence of the Height Theorem is that prime ideals of codimension 1 are principal.

\begin{corollary}
	\label{app:cor:principal}
	If $I\subset\q[X_n]$ is a prime ideal of codimension $1$, then $I$ is principal.
\end{corollary}
\begin{proof}By Krull's Height Theorem $I$ is minimal over the ideal $\ideal f$ generated by some $f\in I$. If $f$ is not irreducible, then it has an irreducible factor $p$ that is in $I$. Therefore, we have
$\{0\}\subset \ideal f \subseteq I\text{ and }\{0\}\subset \ideal f \subseteq \ideal{p}$.
By minimality we must have $I=\ideal{p}$.
\end{proof}

\subparagraph{Algebraic matroid of a prime ideal. } Intuitively, a collection of variables is {\em independent} if it is not constrained by any polynomial in the ideal, and {\em dependent} otherwise. Thus the \emph{algebraic matroid}  $\amat I$ induced by the ideal is, informally, a matroid 
on the ground set of variables $X_n$ whose independent sets are subsets of variables that are {\em not} supported by any polynomial in the ideal.  Its {\em dependent sets} are supports of polynomials in the ideal.

Every algebraic matroid of a prime ideal arises as an algebraic matroid of a field extension in the sense of Definition \ref{def:algmat}, and vice-versa. This equivalence is well-known and we include it for completeness. 

Formally, let $I$ be a prime ideal of the polynomial ring $\q[X_n]$. We define a matroid on $X_n$, depending on the ideal $I$, called \emph{the algebraic matroid of} $I$ and denoted $\amat I$, in the following way.

The quotient ring $\q[X_n]/I$ is an integral domain with a well defined fraction field $K=\ffield{(\q[X_n]/I)}$ which contains $\q$ as a subfield. The image of $X_n$ under the canonical injections
\[\q[X_n]\hookrightarrow \q[X_n]/I\hookrightarrow\ffield{(\q[X_n]/I)}=K\]
is the subset $\{\overline{x_1},\dots,\overline{x_n}\}$ of $K$, where $\overline{x_j}$ denotes the equivalence class of $x_j$ in both $\q[X_n]/I$ and $K$.

Let $X$ be a non-empty subset of $X_n$. Consider its image $\overline X$ in $K$ under the canonical injections. For clarity, let $X=\{x_1,\dots,x_i\}$ and $\overline X=\{\overline{x_1},\dots,\overline{x_i}\}$ for some fixed $i\leq n$. The set $\overline X$ is by definition algebraically dependent over $\q$ if and only if there exists a non-zero polynomial $f\in \q[x_1,\dots,x_i]$ vanishing on $\overline X$, i.e.\
\[f(\overline{x_1},\dots,\overline{x_i})=\overline 0.\]
Clearly, $\overline X$ is algebraically dependent over $\q$ if and only if $f(x_1,\dots,x_i)\in I$, that is if and only if
\[I\cap \q[X]\neq\{0\},\]
where $\q[X]$ denotes the ring of polynomials supported on subsets of $X$.
Similarly, $\overline X$ is algebraically independent over $\q$ if and only if
\[I\cap \q[X]=\{0\}.\]

\begin{definition}
	Let $I$ be a prime ideal in the polynomial ring $\q[X_n]$. The \emph{algebraic matroid of} $I$, denoted $\amat I$, is the matroid $(X_n,\mathcal I)$  on the ground set $X_n=\{x_1,\dots,x_n\}$ with
 	\[\mathcal I=\{X\subseteq X_n\mid I\cap \q[X]=\{0\}\},\]
where $\q[X]$ denotes the ring of polynomials supported on subsets of $X$.	
\end{definition}

\subparagraph{Equivalence of the two definitions} The above construction shows that any algebraic matroid with respect to a prime ideal $I\subset \q[X_n]$ can be realized as an algebraic matroid over $\q$ with the ground set $\{\overline{x_1},\dots,\overline{x_n}\}$ in the field extension $\ffield{(\q[X_n]/I)}$ of $\q$. Conversely, given a set of elements $E=\{\alpha_1,\dots,\alpha_n\}$ in a field extension of $\q$, we can realize any algebraic matroid $\mathcal M$ on $E$ over $\q$ as an algebraic matroid of a prime ideal of $\q[X_n]$ in the following way: let $\varphi$ be the homomorphism
\[\varphi\colon \q[X_n]\to \q(\alpha_1,\dots,\alpha_n)\]
mapping $x_i\mapsto\alpha_i$ for all $i\in\{1,\dots,n\}$ and $a\mapsto a$ for all $a\in \q$. Let $A\subset \{\alpha_1,\dots,\alpha_n\}$ be a dependent set in $\mathcal M$. Then $A$ vanishes on a polynomial in $\q[X_n]$ so the kernel $\ker\varphi$ is non-zero, and clearly any polynomial in $\ker\varphi$ defines a dependency in $\mathcal M$. Therefore, if we denote by $\q[A]$ the ring of polynomials supported on subsets of $\varphi^{-1}(A)$, we have
\[\ker\varphi\cap k[A]\neq{0}\]
if and only if $A$ is a dependent set of $\mathcal M$.

For the rest of the paper we will work exclusively with algebraic matroids of prime ideals.

\subparagraph{Circuits and circuit polynomials.} A {\em circuit} is a minimal set of variables supported by a polynomial in $I$. A polynomial whose support is a circuit is called a {\em circuit polynomial}. A theorem of Lovasz and Dress \cite{DressLovasz} states that a {\em circuit polynomial $p_C$ is unique} in the ideal with the given support $C\subset X_n$, up to multiplication by a constant (we'll just say, shortly, that it is {\em unique}). Furthermore, the circuit polynomial is {\em irreducible}. 

We retain the following property, stating that {\em circuit polynomials generate elimination ideals supported on circuits.} 

\begin{theorem} 
	\label{app:thm:circuitPolyPrincipalIdeal}
	Let $I$ be a prime ideal in $\q[X]$ and $C\subset X$ a circuit of the algebraic matroid $\amat I$. The ideal $I\cap \q[C]$ is principal and generated by an irreducible circuit polynomial $p_C$, which is unique up to multiplication by a constant.
\end{theorem}

\begin{proof}Since $C$ is a circuit the ideal $I\cap \q[C]$ has dimension at least equal to $|C|-1$. It can not have dimension greater than or equal to $|C|$ since $\dim k[C]=|C|$. Therefore $I\cap \q[C]$ has dimension $|C|-1$ and codimension $1$. By Corollary \ref{app:cor:principal}, \  $I\cap \q[C]$ is principal, and since $p_C\in I\cap \q[C]$ is irreducible over $\q$, it also generates $I\cap \q[C]$.\end{proof}

\section{The Cayley-Menger ideal and its algebraic matroid}
\label{sec:prelimCMideal}

In this section we introduce the 2D Cayley-Menger ideal $\cm n$. We then define the corresponding {\em circuit polynomials} and their supports, and make the connection with combinatorial rigidity circuits.

We will show that the algebraic matroid of $\cm n$ is isomophic to the $(2,3)$-sparsity matroid $\smat n$. This equivalence is well-known, however we were not able to track down the original reference, and include a proof for completeness. 


\subparagraph{The Cayley-Menger  ideal and its algebraic matroid.} We use variables $X_n = \{x_{i,j}: 1\leq i<j\leq n\}$ for unknown squared distances between pairs of points. The {\em distance matrix} of $n$ labeled points is the matrix of squared distances between pairs of points. The {\em Cayley matrix} is the distance matrix bordered by a new row and column of 1's, with zeros on the diagonal: 

\vspace{-20pt}
\begin{center}
$$
\begin{pmatrix}
	0 & 1 & 1 & 1 & \cdots & 1\\
	1 & 0 & x_{1,2} & x_{1,3} & \cdots & x_{1,n}\\
	1 & x_{1,2} & 0 & x_{2,3} & \cdots & x_{2,n}\\
	1 & x_{1,3} & x_{2,3} & 0 & \cdots & x_{3,n}\\
	\vdots & \vdots & \vdots &\vdots &\ddots &\vdots\\
	1 & x_{1,n} & x_{2,n} & x_{3,n} & \cdots & 0
\end{pmatrix}
$$
\end{center}

\noindent
Cayley's Theorem says that, if the distances come from a point set in the Euclidean space $\reals^d$, then the rank of this matrix must be at most $d+2$. Thus all the $(d+3)\times(d+3)$ minors of the Cayley-Menger matrix should be zero.  These minors induce polynomials in $\q[X_n]$ which generate the $(n,d)$-Cayley-Menger ideal. They are called the \emph{standard generators}, are {\em homogeneous polynomials} with integer coefficients and are {\em irreducible} over $\q$. The $(n,d)$-Cayley-Menger ideal is a {\em prime ideal} of dimension $dn-{\binom{d+1}{2}}$ \cite{borcea:cayleyMengerVariety:2002, Giambelli, HarrisTu, JozefiakLascouxPragacz} and codimension $\binom{n}{2}-dn+{\binom{d+1}{2}}$.
We work with the {\em 2D Cayley-Menger ideal} $\cm n$, generated by the $5\times 5$ minors of the Cayley matrix. The algebraic matroid of the 2D Cayley-Menger ideal $\cm n$ is denoted by $\amat {\cm n}$.

\subparagraph{The algebraic matroid of $\cm{n}^d$.} As defined in Section \ref{sec:prelimAlgMatroids}, the algebraic matroid $\amat{\cm{n}^d}$ is the matroid on the ground set $X_n=\{x_{i,j}\mid 1\leq i<j\leq n\}$ where $X\subseteq X_n$ is independent if
\[\cm{n}^d\cap~\q[X]=\{0\},\]
where, as before, $\q[X]$ denotes the ring of polynomials over $\q$ supported on the indeterminates in $X$. 

\begin{proposition}The rank of $\amat{\cm{n}^d}$ is equal to $\dim{\cm{n}^d}=dn-{\binom{d+1}{2}}$.\end{proposition}

\begin{proof} Immediate from the definition of dimension of an ideal in a ring of polynomials (Section~\ref{sec:prelimAlgMatroids}).
\end{proof}

When $d=2$ the rank of $\cm{n}^2$ is precisely the rank of the $(2,3)$-sparsity matroid $\smat n$ on $n$ vertices. In fact, the two matroids are isomorphic. For the rest of the paper we fix $d=2$ and abbreviate $\cm{n}^2$ with just $\cm{n}$.

\begin{theorem}
	\label{theorem:algSparseIsomorphic}
	The algebraic matroid $\amat{\cm{n}}$ and the $(2,3)$-sparsity matroid $\smat n$ are isomorphic.
\end{theorem}

We will prove the equivalence of $\amat{\cm{n}}$ and $\smat n$ by proving that both are equivalent to the $2$-dimensional generic linear rigidity matroid that we now introduce.

Let $G=(V,E)$ be a graph and $(G,p)$ a $2D$ bar-and-joint framework with points $\{p_1,\dots,p_n\}\subset\mathbb R^2$.

\begin{definition}The rigidity matrix $R_{(G,p)}$ (or just $R_G$ when there is no possibility of confusion) of the bar-and-joint framework $(G,p)$ is the $|E|\times n$ matrix with columns indexed by the vertices $\{1,2,\dots,n\}$ and rows indexed by the edges $ij\in E$ with $i<j$. The $i$-th entry in the row $ij$ is $p_i-p_j$, the $j$-th entry is $p_j-p_i$, and all other entries are $0$.
\end{definition}

The rigidity matrix is defined up to an order of the vertices and the edges; to eliminate this ambiguity we fix the order on the vertices as $1<2<\cdots<n$ and we order the edges $ij$ with $i<j$ lexicographically. 

For example, let $G=K_4$. Then the rows are ordered as $12$, $13$, $14$, $23$, $24$ and $34$ and the corresponding rigidity matrix $R_{K_4}$ is given by
\[R_{K_4}=\begin{pmatrix}
	p_1-p_2 & p_2 - p_1 & 0 & 0\\
	p_1-p_3 & 0 & p_3-p_1 & 0\\
	p_1-p_4 & 0 & 0 & p_4 - p_1\\
	0 & p_2-p_3 & p_3-p_2 & 0\\
	0 & p_2-p_4 & 0 & p_4 - p_2\\
	0 & 0 & p_3-p_4 & p_4-p_3\\
\end{pmatrix}.\]

\begin{definition}Let $G=(V,E)$ be a graph and $(G,p)$ a $2D$ bar-and-joint framework. The \emph{2-dimensional linear rigidity matroid} $\lmat{(G,p)}$ is the matroid $(E,\mathcal I)$ where $I\in\mathcal I$ if and only if the rows of $R_{(G,p)}$ indexed by $I$ are linearly independent.
\end{definition}

Note that the matroid $\lmat{(G,p)}$ depends on the plane configuration $p$. For example, if $G=K_4$, $p$ a configuration in which at most two vertices of $K_4$ are on a line, and $q$ a configuration in which the vertices $\{2,3,4\}$ are on the same line, then $\rank\lmat{K_4,p}>\rank\lmat{K_4,q}$.

Let $G$ be a graph and consider the set $\mathcal X_G$ of all possible plane configurations $p$ for $G$.  

\begin{definition} We say that a 2D bar-and-joint framework $(G,p)$ is \emph{generic} if the rank of the row space of $R_{(G,p)}$ is maximal among all plane configurations in $\mathcal X_G$.
\end{definition}

If $p$ and $p'$ are distinct generic plane configurations for a graph $G$, the 2D linear matroids $\lmat{(G,p)}$ and $\lmat{(G,p')}$ are isomorphic \cite[Theorem 2.2.1]{graverServatiusServatius}, hence the following matroid is well-defined.

\begin{definition}Let $G$ be a graph. The \emph{2-dimensional generic linear matroid} $\lmat{G}$ is the 2D linear matroid $\lmat{(G,p)}$ where $p$ is any generic plane configuration for $G$.
\end{definition}

That for a given graph $G$ on $n$ vertices the generic linear matroid $\lmat{(G,p)}$ and the sparsity matroid $\smat n$ are isomorphic follows from Laman's theorem (Theorem \ref{thm:laman}).


We now have to show that $\amat{\cm{n}}$ is equivalent to the generic linear rigidity matroid $\lmat{K_{n}}$. This equivalence will be the consequence of a classical result of Ingleton \cite[Section 6]{Ingleton} (see also \cite[Section 2]{EhrenborgRota}) stating that algebraic matroids over a field $k$ of characteristic zero are linearly representable over an extension of $k$, with the linear representation given by the Jacobian. We now note that the Cayley-Menger variety is realized as the Zariski closure of the image of the map $f=(f_{ij})_{\{i,j\}\in{\binom{n}{2}}}\colon(\mathbb C^2)^n \to \mathbb C^{\binom{n}{2}}$ given by 
\[(p_1,\dots,p_n)\mapsto(||p_i-p_j||^2)_{\{i,j\}\in{\binom{n}{2}}}.\]
The Jacobian of the edge function $(p_1,\dots,p_n)\mapsto(||p_i-p_j||^2)_{\{i,j\}\in{\binom{n}{2}}}$ at a generic point in $(\mathbb C^2)^n$ is precisely the matrix $2R_{(K_n,p)}$ for a generic configuration $p$.

{\bf This completes the proof of Theorem~\ref{theorem:algSparseIsomorphic}.} From now on, we will use the isomorphism to move freely between the formulation of algebraic circuits as subsets of variables $X\subset X_n$ and their graph-theoretic interpretation as graphs that are rigidity circuits. Given a (rigidity) circuit $C$, we denote by $p_C$ the corresponding {\em circuit polynomial} in the Cayley-Menger ideal $\cm n$.

\subparagraph{Comments: beyond dimension 2?}
Note that the $d$-dimensional linear rigidity matroid $\lmat n$ and the algebraic matroid $\amat{\cm{n}^d}$ of the $(n,d)$-Cayley-Menger matroid are isomorphic by the same Jacobian argument as above. However, the equivalence between the 2D sparsity matroid $\smat n$ and $\amat{\cm{n}}$ does not extend to higher dimension to some known graphical matroid. The generalization $dn - \binom{d+1}{2}$  of the $(2n-3)$-sparsity condition from dimension $2$ to dimension $d$, called Maxwell's sparsity \cite{maxwell:equilibrium:1864}, does not satisfy matroid axioms. In fact, Maxwell's sparsity is known to be only a necessary but not sufficient condition for minimal rigidity in dimensions $d\geq 3$.

\subparagraph{Circuits of $\amat{\cm n}$ and circuit polynomials in $\cm n$.} The isomorphism between the algebraic matroid $\amat {\cm n}$ and the sparsity matroid $\smat n$ immediately implies that the sets of circuits of these two matroids are in a one-to-one correspondence. We will identify a sparsity circuit $C=(V_C,E_C)\in \smat n$, with the circuit
\[\{x_{i,j}\mid ij\in E_C\}\in\amat{\cm n},\]
and in the same way the dependent sets of $\smat n$ will be identified with the dependent sets of $\amat{\cm n}$. Conversely, we will identify the support of a polynomial $f\in \q[\{x_{i,j}\mid 1\leq i<j\leq n\}]$ with the graph $S_f=(V_f,E_f)$ where
\[V_f=\{i\mid x_{i,j}\text{ or }x_{h,i}\in \supp f\}\text{ and }E_f=\{ij\mid x_{i,j}\in\supp f\}.\]

Recall that the circuit polynomial of a circuit $C$ in $\cm n$ is the (up to multiplication with a unit) unique polynomial $p_C$ irreducible over $\mathbb Q$ such that $\supp{p_C}=C$. Hence we will identify from now on a circuit $C$ with the support $\supp p_C$ of its circuit polynomial $p_C$. Furthermore, $p_C$ generates the elimination ideal $\cm n\cap \q[C]$.

\begin{proposition}\label{prop:circHomogeneous}Circuit polynomials in $\cm n$ are homogeneous polynomials.
\end{proposition} 
\begin{proof}Since $\cm n$ is generated by homogeneous polynomials, any reduced \grobner{} basis of $\cm n$ consists only of homogeneous polynomials (see e.g.\ Theorem 2 in \S3 of Chapter 8 of \cite{CoxLittleOshea}). Now if $C$ is a circuit in $\cm n$, we can choose an elimination order in which all the indeterminates in the complement of $C$ are greater than those in $C$. The \grobner{} basis $\mathcal G_C$ with respect to that elimination order will necessarily contain $p_C$ because $\mathcal G_C\cap \q[C]$ must generate the elimination ideal $\cm n\cap \q[C]$.\end{proof}

By Theorem \ref{theorem:algSparseIsomorphic} any circuit polynomial in $\cm n$ is supported on a sparsity circuit. Therefore, the simplest examples of circuit polynomials in $\cm n$ are the standard generators that are supported on the $\binom n4$ complete graphs on 4 vertices in $[n]$.
 

{\bf Example: the $K_4$ circuit.} Among the generators of $\cm n$ we find the smallest circuit polynomials. Their supports are in correspondence with the edges of complete graphs $K_4$ on all subsets of $4$ vertices in $[n]$. The $K_4$ circuit polynomial (given below on vertices $1234$) is homogeneous of degree 3, has 22 terms and has degree 2 in any of its variables.
	\begin{align*}
	p_{K_4^{1234}}&=x_{3,4} x_{1,2}^2+x_{3,4}^2 x_{1,2}+x_{1,3} x_{2,3} x_{1,2}-x_{1,4} x_{2,3} x_{1,2}-x_{1,3} x_{2,4} x_{1,2}\\
	&+x_{1,4}^2 x_{2,3}+x_{1,3} x_{2,4}^2+x_{1,4} x_{2,4} x_{1,2}-
	x_{1,3} x_{3,4} x_{1,2}-x_{1,4} x_{3,4} x_{1,2}\\
	&+x_{1,3}^2 x_{2,4}+x_{1,4} x_{2,3}^2-x_{2,3} x_{3,4} x_{1,2}-x_{2,4} x_{3,4} x_{1,2}
	+x_{2,3} x_{2,4} x_{3,4}\\
	&-x_{1,3} x_{2,4} x_{3,4}-x_{1,3} x_{1,4} x_{2,3}
	-x_{1,3} x_{1,4} x_{2,4}-x_{1,3} x_{2,3} x_{2,4}\\
	&-x_{1,4} x_{2,3} x_{2,4}+x_{1,3} x_{1,4} x_{3,4}-x_{1,4} x_{2,3} x_{3,4}
	\end{align*}

\subparagraph{Cayley-Menger variety. } The $(n,d)$-Cayley-Menger variety can also be obtained from the ideal generated by the $(d+1)\times(d+1)$ minors of the symmetric $(n-1)\times(n-1)$ \emph{Gram matrix} $[\text{Gram}_{i,j}]$ on $n$ points defined as follows: we first pick one of the points $\{p_1,\dots,p_n\}$ as a reference point, say $p_n$ and set
\[\text{Gram}_{i,j}=\frac12(x_{i,n}+x_{j,n}-x_{i,j})\text{ for }i,j\in\{1,\dots,n-1\}\]
with $x_{i,i}=0$ so that $\text{Gram}_{i,i}=x_{i,n}$.

For our purposes, it is preferable to work with the Cayley-Menger matrix and the standard generators as there will be no need for a reference point.


\subparagraph{Resultants in the Cayley-Menger ideal.} Let $f,g$ be two polynomials in the Cayley-Menger ideal with $x_{ij}$ one of their common variables. We treat them as polynomials in $x_{ij}$, therefore the coefficients are themselves polynomials in the remaining variables. Our {\em main observation}, which motivated the  definition of the combinatorial resultant, is that the entries in the Sylvester matrix are polynomials supported exactly on the variables corresponding to the {\em combinatorial resultant} of the supports of $f$ and $g$ on elimination variable (edge) $ij$. 

The following lemma, whose proof follows immediately from Proposition \ref{prop:resultantElimination}, will be used frequently in the rest of the paper. 

\begin{lemma}
	\label{lem:resultantSupport}
	Let $I$ be an ideal in $\q[X_n]$, $f,g\in I$ with supports $S_f= \supp f$ and $S_g= \supp g$, and let $x_{ij}$ be a common variable in $S_f\cap S_g$. Let $S$ be the combinatorial resultant of the supports $S=\cres{S_f}{S_g}{ij}\subset X_n$ (i.e. viewed as a set of variables). Then $\res{f}{g}{x_{ij}}\in I\cap \q[S]$.
\end{lemma}

\subparagraph{Homogeneous properties.} The standard generators of the Cayley-Menger ideal, in particular those that correspond to $K_4$ graphs, are obviously homogeneous. We apply Proposition~\ref{prop:resultantHomogeneous} to infer that their resultants are themselves homogeneous polynomials and to compute their homogeneous degrees. 


\section{Resultants of circuit polynomials}
\label{sec:algResCircuits}

We are now ready to prove our main result, Theorem~\ref{thm:circPolyConstruction}, by showing that combinatorial resultants are the combinatorial analogue of classical polynomial resultants in the following sense: if a (rigidity) circuit $C$ is obtained as the combinatorial resultant $\cres{A}{B}{e}$ of two circuits $A$ and $B$ with the edge $e$ eliminated, then the resultant $\res{p_{A}}{p_{B}}{x_e}$ of circuit polynomials $p_{A}$ and $p_{B}$ with respect to the indeterminate $x_e$ is supported on $C$ and contained in the elimination ideal $\ideal{p_C}$ generated by the circuit polynomial $p_C$. When $\res{p_{A}}{p_{B}}{x_e}$ is irreducible then it will be equal to $p_C$. However in general $p_C$ will only be one of its irreducible factors over $\mathbb Q$. In fact {\em exactly one factor} (counted with multiplicity) of $\res{p_{A}}{p_{B}}{x_e}$ can correspond to $p_C$ and that factor can be deduced by examining the supports of the factors and performing an ideal membership test on those factors that have the support of $p_C$. Our approach is algorithmic, but its precise complexity analysis depends on answers to a few remaining open questions:

\begin{problem}
Identify sufficient conditions under which $\res{p_{A}}{p_{B}}{x_e}$ is $p_C$.
\end{problem}

\begin{problem}
Identify sufficient conditions for which $\res{p_{A}}{p_{B}}{x_e}$ has exactly one factor (up to multiplicity) supported on $C$. 
\end{problem}
	
\subparagraph{Resultants of circuit polynomials.} In this section we refer to a (combinatorial) circuit as a {\em sparsity} circuit, to avoid confusion with the usage of the same concept (circuit) in the {\em circuit polynomial}.

\begin{lemma}
	\label{lem:nonZeroCircuitPoly}Let $A$ and $B$ be distinct sparsity circuits. The resultant $\res{p_{A}}{p_{B}}{x_e}$ of the circuit polynomials $p_{A}$ and $p_{B}$ with respect to a variable $x_e\in \supp{p_{A}}\cup~\supp{p_{B}}$ is a non-constant polynomial in $\mathbb Q[(\supp{p_{A}}\cup~\supp{p_{B}})\setminus\{x_e\}]$.
\end{lemma}

\begin{proof} Circuit polynomials are irreducible over $\mathbb Q$, hence the resultant of two distinct circuit polynomials is non-zero (by Proposition
\ref{prop:resultantCommonFactor}
in the Appendix).
\end{proof}

Let $A$ and $B$ be distinct sparsity circuits and let $p_{A}$ and $p_{B}$ be the corresponding circuit polynomials. Assuming $e\in A\cap B$, the polynomials $p_{A}$ and $p_{B}$ belong to the ring $R[x_e]$, where
$R=\mathbb Q[(\supp{p_{A}}\cup\supp{p_{B}})\setminus\{x_e\}].$

\begin{theorem}\label{lem:structureCircuitPoly}Let $C$ be a sparsity circuit on $n+1$ vertices and $p_C$ its corresponding circuit polynomial. There exist sparsity circuits $A$ and $B$ on at most $n$ vertices with circuit polynomials $p_{A}$ and $p_{B}$ such that $p_{C}$ is an irreducible factor over $\mathbb Q$ of $\res{p_{A}}{p_{B}}{x_e}$, where $e\in A\cap B$.
\end{theorem}

\begin{proof}Given a sparsity circuit $C$ on $n+1$ vertices we can find two sparsity circuits $A$ and $B$ on at most $n$ vertices such that $C=\cres{A}{B}{e}$ for some $e\in A\cap B$ by the proof of Theorem \ref{thm:combResConstruction} in Section~\ref{sec:combRes}. Let $p_{A}$ and $p_{B}$ be the corresponding circuit polynomials.

The polynomials $p_{A}$ and $p_{B}$ are contained in $\cm{m}$ for some $m\geq n+1$ and the resultant $\res{p_{A}}{p_{B}}{x_e}$ is a non-constant polynomial in $R=\mathbb Q[(\supp{p_{A}}\cup\supp{p_{B}})\setminus\{x_e\}]$ supported on $\supp{p_C}$. Since $\ideal{p_{A},{p_{B}}}\subset\cm{m}$, we have that $\res{p_{A}}{p_{B}}{x_e}$ is contained in the elimination ideal $\cm{m}\cap\mathbb{Q}[\supp{p_C}]=\ideal{p_C}$ (by Lemma \ref{lem:resultantSupport}).
\end{proof}

\begin{corollary}Under the assumptions of Theorem \ref{lem:structureCircuitPoly}, the resultant $\res{p_{A}}{p_{B}}{x_e}$ is a circuit polynomial if and only if it is irreducible (over $\mathbb Q$).\end{corollary}

Indeed, it is not always the case that the resultant of two circuit polynomials, such that its support corresponds to a sparsity circuit is irreducible. Recall from Corollary \ref{cor:nonunique} that in general a sparsity circuit $C$ can be represented as the combinatorial resultant of two circuits in more than one way. If $C=\cres{C_1}{C_2}{e}=\cres{C_3}{C_4}{f}$ and $p_{C_i}$ for $i\in\{1,\dots,4\}$ are the corresponding circuit polynomials, then $\res{p_{C_1}}{p_{C_2}}{x_e}$ and $\res{p_{C_3}}{p_{C_4}}{x_f}$ will in general be distinct elements of $\ideal{p_C}$. Moreover, the resultant with the higher homogeneous degree must necessarily have a non-trivial factor.

\begin{example}
	\label{example:resultantNotIrreducible}
	Let $C$ be the 2-connected circuit $\{12,13,23,24,34,15,16,56,45,46\}$ as shown in Figure \ref{fig:exampleCircuit}. This circuit can be obtained as the combinatorial resultant of the circuits $C_1$ and $C_2$ given by the $K_4$ graph on $1234$ and $1456$, respectively, with the edge $14$ eliminated (Figure \ref{fig:exampleCircuit}, left). It can also be obtained as the combinatorial resultant  of the circuits $C_3$ and $C_4$ given by the wheels on $1245$ with $3$ in the center and $1346$ with $5$ in the center, with the edge $35$ eliminated (Figure \ref{fig:exampleCircuit}, right). However, $\res{p_{C_1}}{p_{C_2}}{x_{1,4}}$ is the resultant of two quadratic polynomials of homogeneous degree 3, whereas $\res{p_{C_3}}{p_{C_4}}{x_{3,5}}$ is the resultant of two quartic polynomials of homogeneous degree 8. Therefore, by Proposition \ref{prop:resultantHomogeneous}, $\res{p_{C_1}}{p_{C_2}}{x_{1,4}}$ is homogeneous of degree 8, whereas $\res{p_{C_3}}{p_{C_4}}{x_{3,5}}$ is of homogeneous degree 48. Both resultants have the same circuit as its supporting set, hence they are both in the elimination ideal $\ideal{p_C}$ of $\cm{6}$ which is generated by $\res{p_{C_1}}{p_{C_2}}{x_{1,4}}$. Therefore $\res{p_{C_1}}{p_{C_2}}{x_{1,4}}=\alpha\cdot p_C$ and $\res{p_{C_3}}{p_{C_4}}{x_{3,5}}=q\cdot p_C$ where $\alpha\in\mathbb Q$ and $q$ is a non-constant polynomial in $\mathbb Q[C]$. \end{example}

We can generalize Example \ref{example:resultantNotIrreducible} in the following way. Let $C$ be a sparsity circuit on $n\geq 5$ vertices. Consider the set $\Gamma_C$ of all possible representations of $C$ as a combinatorial resultant of two sparsity circuits $A$ and $B$ on at most $n$ vertices 
\[\Gamma_C=\{(A,B,e)\mid C=\cres{A}{B}{e}, |V(A)|,|V(B)|\leq |V(C)|\}\]
and the set
\[\rho_C=\{\res{p_{A}}{p_{B}}{x_e}\mid (A,B,e)\in\Gamma_C\}\]
of all resultants of corresponding circuit polynomials. The circuit polynomial $p_C$ of the circuit $C$  of Example \ref{example:resultantNotIrreducible} had the property of being the polynomial in $\rho_C$ of minimal homogeneous degree. One might therefore conjecture that for any sparsity circuit $C$, the polynomial in $\rho_C$ of minimal homogeneous degree is the circuit polynomial for $C$; in that case no irreducibility check would be required as we can compute the homogeneous degree of $\res{p_{A}}{p_{B}}{x_e}$ from the homogeneous degrees and the degrees in $x_e$ of $p_{A}$ and $p_{B}$ (Proposition \ref{prop:resultantHomogeneous}). However, we will show in Proposition \ref{prop:k33notObtainable} that in general the circuit polynomial of a circuit $C$ is not necessarily in $\rho_C$. This fact leads to the following natural question.

\begin{problem}Let $C$, $A$ and $B$ be sparsity circuits such that $C=\cres{A}{B}{e}$ with $p_C$, $p_{A}$ and $p_{B}$ the corresponding circuit polynomials. Under which conditions on $A$, $B$, $p_{A}$ and $p_{B}$ is $\res{p_{A}}{p_{B}}{x_e}=\alpha\cdot p_C$ for some $\alpha\in\mathbb Q$?
\end{problem}

\subparagraph{Algorithms: determining the circuit polynomial from the resultant.}
If $\res{p_{A}}{p_{B}}{x_e}$ is not irreducible, then (up to multiplicity) exactly one of its irreducible factors (over $\mathbb Q$) is in $\cm n$, and that factor is precisely the circuit polynomial $p_C$. This factor can be deduced in two steps: an analysis of the supports of the factors and an ideal membership test. 

\textbf{Step 1: analysing the supports of the irreducible factors.} If $C=\cres{A}{B}{e}$ then we know that $C$ and $\supp p_C$ are identified. Recall that the elimination ideal $\ideal{p_C}$ is an ideal of $\mathbb \q[C]$, and since $\res{p_{A}}{p_{B}}{x_e}\in\ideal{p_C}$, any irreducible factor (over $\mathbb Q$) of this resultant is supported on a subset of $\supp p_C$ that is not necessarily proper. At least one these factors must be supported on exactly $\supp p_C$, and if there is only one such factor, then that factor must be $p_C$. Otherwise, we proceed to Step 2.

\textbf{Step 2: ideal membership test.} Take into consideration only those irreducible factors of $\res{p_{A}}{p_{B}}{x_e}$ that are supported on $\supp{p_C}$. Test each factor for membership in $\cm n$ via a \grobner{} basis algorithm with respect to some monomial order, not necessarily an elimination order. The first factor determined to be in $\cm n$ is $p_C$.

\begin{algorithm}[ht]
	\caption{Circuit polynomial: resultant step}
	\label{alg:resultant}
	\textbf{Input}: Circuits $A$, $B$ and edge $e$ such that $C=\cres{A}{B}{e}$. Circuit polynomials $p_A$ and $p_B$ and elimination variable $x_e$.\\
	\textbf{Output}: Circuit polynomial $p_C$ for $C$.
	\begin{algorithmic}[1]
		\State Compute the resultant $p$ = Res($p_A$,$p_B$,$x_e$).
		\If {$p$ is irreducible} 
		\State $p_C = p$
		\Return $p_C$
		\Else
		\State factors = factorize $p$ over $\q$
		\State factors = discard factors with support not equal to $C$
		\If {exactly one remaining factor (possibly with multiplicity)}
		\State$p_C$ = the unique factor supported on $C$
		\Else 
		\State apply a test of membership in the $CM$ ideal on the remaining factors
		\State $p_C$ = unique factor for which ideal membership test succeeded
		\State \Return $p_C$		
		\EndIf
		\EndIf		
	\end{algorithmic}
\end{algorithm}

Every circuit polynomial can inductively be computed from the circuit polynomials supported on complete graphs on $4$ vertices by repeatedly calculating resultants and applying steps 1 and 2. Hence we can preform all the computations in a prime ideal that is minimal over an ideal generated only by circuit polynomials of complete graphs on 4 vertices. In order to define this ideal we first define the following tree.

Let $C$ be a sparsity circuit on $n$ vertices and $T_C$ a resultant tree for $C$. Denote by $\supp T_C$ the union of the supports of all leaves in $T_C$:
\[\supp T_C=\{x_{i,j}\mid \{i,j\}\text{ is an edge of a leaf in }T_C\}.\]Consider in $S_{T_C}=\mathbb{C}[\supp T_C]$ the ideal
\[
K_{T_C} =\langle\{p_{K}\text{ is a circuit polynomial in}\cm n\mid
 K\text{ is a leaf of }T_C\}\rangle
\]
generated by all circuit polynomials supported on the complete graphs on $4$ vertices that appear in $T_C$. This ideal need not be prime (see Proposition \ref{prop:k33notObtainable}), however it always has at least one minimal prime ideal $P_{T_C}$ above it.

\begin{theorem}
	\label{thm:minPrime}
	Let $C$ be a sparsity circuit on $n$ vertices, $T_C$ a resultant tree for $C$ and $p_C\in\cm n$ its circuit polynomial. If $P_{T_C}$ is a minimal prime ideal over $K_{T_C}$, then $p_C\in P_{T_C}$. 
\end{theorem}

\begin{proof}By assumption the circuit polynomials of the leaves of $T_C$ are in $P_{T_C}$. By computing the resultants along the tree $T_C$, at each node $C_i$ we obtain a polynomial $r_i\in P_{T_C}$ that has the circuit polynomial $p_{C_i}$ of $C_i$ as a factor. Since $P_{T_C}$ is a prime ideal and $P_{T_C}\cap\mathbb Q[C_i]\subseteq \cm n\cap \mathbb Q[C_i]=\ideal{p_{C_i}}$ it follows that the only irreducible factor of $r_i$ that can be in $P_{T_C}$ is $p_{C_i}$. Therefore $p_C\in P_{T_C}$.\end{proof}

Note that we can either perform a factorization and a membership test in the ideal $P_{T_C}$ at each node, or we can can preform these two operations only once at the root of $T_C$ because of the multiplicativity of resultants (Proposition \ref{prop:basicPropResultants} (ii)). In both cases we have to compute a minimal prime ideal over $K_{T_C}$ which requires a \grobner{} basis computation (see \cite{DeckerGreuelPfister} for a survey of methods for computing prime ideals minimal over a given ideal), although both options have their advantages and drawbacks:
\begin{itemize}
	\item by factorizing the resultant $r_i$ and applying an ideal membership test to its factors at each node $C_i$, we can discard unnecessary factors and carry to the parent level only the circuit polynomial $p_{C_i}$, reducing the complexity of the subsequent computation of the resultant in the parent node of $C_i$. The drawback is that we have to preform a factorization at every non-leaf node, and a membership tests at every node at which the resultant has distinct factors.
	\item By choosing to forgo factorization and an ideal membership test at each node except the root, we have to perform only one factorization and only one ideal membership test at the root of $T_C$ for the resultant $\res{A}{B}{x_e}$ with respect to some common indeterminate $x_e$, where $A$ and $B$ themselves are resultants of resultants obtained in the previous levels. However, the drawback is that not just the factorization, but even the computation of $\res{A}{B}{x_e}$ itself becomes very difficult.
\end{itemize}

\subparagraph{Comments on complexity.} The main bottleneck in our approach for computing circuit polynomials is that we still have to compute a \grobner{} basis with respect to some monomial order in order to apply an ideal membership test. It is difficult to know a priori which monomial order will be the most efficient. However in practice elimination orders, necessary for the general method for computing elimination ideals, behave badly (see \cite[Section 4]{BayerMumford} and the Section \emph{Complexity Issues} in \cite[\S10 of Chapter 2]{CoxLittleOshea}) while graded orders show better performance but cannot be used to compute elimination ideals.

Our approach however avoids the use of an elimination order, requires only one elimination step that is obtained with resultants, and is followed by a factorization with a potential ideal membership test that can be preformed with respect to \emph{any} monomial order. Hence we are free to choose a monomial order for $\cm n$ that we expect to have better performance than an elimination order.

A measure of complexity of a \grobner{} basis computation is given by bounds on the total degrees\footnote{The total degree of a polynomial is the maximum of the homogeneous degrees of its monomials} of polynomials in the basis, with an exponential lower bound and a doubly exponential upper bound in the number of indeterminates of the underlying ring (see the recent survey \cite{MayrToman}). In particular, Dub\'e \cite{Dube} gives the following bound.

\begin{theorem}(\cite{Dube}) Let $I=\ideal{f_1,\dots,f_s}$ be an ideal of $R[x_1,\dots,x_n]$, where $R$ is a ring and $s,n\geq 1$. If $d$ is the maximum of the total degrees of polynomials in $\{f_1,\dots,f_s\}$, then every reduced \grobner{} basis $B$ of $I$ with respect to an monomial order consists of polynomials $g\in B$ satisfying
\[\deg(g)\leq 2\left(\frac{d^2}{2}+d\right)^{2^{n-1}}.\]\end{theorem}
With respect to $\cm n$ for $n\geq 5$ the standard generators are of homogeneous degree 3, 4 and 5, therefore this bound is $2(25/2+5)^{2^{\binom n2 -1}}$.

For large $n$ there will be circuits $C$ on $n$ vertices that have a resultant tree $T_C$ such that $|\supp T_C|$ is relatively small compared to $\binom n2$. In that case it can be more feasible to preform \grobner{} basis computations in the lower-dimensional ring $\mathbb Q[\supp T_C]$ than to compute a \grobner{} basis for $\cm n$ in the ring on $\binom n2$ indeterminates; first we have to compute a \grobner{} basis for $K_{T_C}$ which has Dub\'e's bound reduced to  $2(9/2+3)^{2^{|\supp T_C| -1}}$, and use that basis to compute a \grobner{} basis for $P_{T_C}$.

Mayr and Ritscher \cite{MayrRitscher} have shown that the total degrees of the polynomials in the \grobner{} basis $B$ are bounded  doubly-exponentially in the dimension of $I$.

\begin{theorem}(\cite{MayrRitscher}) Let $I=\ideal{f_1,\dots,f_s}$ be an ideal of $k[x_1,\dots,x_n]$, where $k$ is an infinite field and $s,n\geq 1$. If $d$ is the maximum of the total degrees of polynomials in $\{f_1,\dots,f_s\}$, then every reduced \grobner{} basis $B$ of $I$ with respect to an monomial order consists of polynomials $g\in B$ satisfying
\[\deg(g)\leq 2\left(\frac 12\left(d^{2(\codim I)^2}+d\right)\right)^{2^{\dim I}}.\]
\end{theorem}
Thus a strategy of computing a circuit polynomial of a circuit $C$ on $n$ vertices is to find the ideal $P_{T_C}$ of least possible dimension (which cannot be smaller than $2n-3$). If the dimension of $P_{T_C}$ is the best possible, i.e.\ equal to $2n-3$, then its codimension is also smaller than the codimension of $\cm n$ because the underlying ring of $P_{T_C}$ has less than $\binom n2$ indeterminates. Therefore, when $\dim P_{T_C}=2n-3$ the Mayr-Ritscher bound for $P_{T_C}$ will be smaller than the bound for $\cm n$ even though we also have $\dim{\cm n}=2n-3$. However, determining the dimension of an ideal in general requires polynomial space \cite{BernasconiMayrMnukRaab}.

\section{Experiments}
\label{sec:experiments}

\subparagraph{Computation of circuit polynomials via \grobner{} bases.} In principle a circuit polynomial $p\in\cm n$ can be computed by computing a \grobner{} basis
$\mathcal G_{\cm n}$ for $\cm n$ with respect to an \emph{elimination order}\footnote{See Excercises 5 and 6 in \S1 of Chapter 3 in \cite{CoxLittleOshea}} on the set $\{x_{i,j}\mid 1\leq i<j\leq n\}$ in which all the indeterminates in $\supp p$ are greater than all the indeterminates in its complement. 

Given $\mathcal G_{\cm n}$ it is straightforward to determine a \grobner{} basis $\mathcal G_{\ideal p}$ for the ideal $\ideal p=\cm{n}\cap\mathbb C[\supp p]$: it is the intersection
\[\mathcal G_{\ideal p}=\mathcal G_{\cm n}\cap \mathbb C[\supp p].\]
Therefore, the only element in $\mathcal G_{\cm n}$ supported on $\supp p$ is precisely $p$, possibly multiplied by a non-zero scalar.

The complexity and performance of algorithms for computing \grobner{} bases depends heavily on the choice of a monomial order. The method described above holds only for elimination orders which in practice often quickly become infeasible. In general, the main problems of Elimination Theory, such as the Ideal Triviality Problem, the Ideal Membership Problem for Complete Intersections, the Radical Membership Problem, the
General Elimination Problem, and the Noether Normalization are in the PSPACE compelxity class \cite{MateraMariaTurull}. 

In the particular case of Cayley-Menger ideals $\cm n$, already for $n=6$ we were not able to compute\footnote{The two  computers that we used scored 2.43 and 2.92 on Mathematica's V12.1.1.0 Benchmark Report, considerably higher than the highest score of 1.89 that the report is comparing against.} a \grobner{} basis with respect to an elimination order, neither within Mathematica nor within Macaulay2. Nevertheless, in the next Section~\ref{sec:algResCircuits} we will present a more efficient method for computing circuit polynomials that allowed us to compute all circuit polynomials for $\cm 6$.

\subparagraph{Circuit polynomials in $\cm 6$.}
We demonstrate now the effectiveness of our method for computing circuit polynomials by computing all the circuit polynomials in $\cm{6}$. These polynomials are supported on six types of graphs: a $K_4$, a wheel on $4$ vertices, a wheel on $5$ vertices, a $2$-dimensional ``double banana'', the Desargues-plus-one-edge graph, and the $K_{3,3}$-plus-one-edge graph. These graphs are shown in Fig.~\ref{fig:graphW5K4} and Fig.~\ref{fig:6circuits}. To the best of our knowledge, except for the circuit polynomial of a $K_4$ graph which is the unique minor generating $\cm 4$, these circuit polynomials have not been computed before.



\subparagraph{The $K_4$ circuits.} The circuit polynomial for the $K_4$ graph on the vertices $1234$ is the (up to a multiplication with a scalar) unique generator of $\cm{4}$ given by
\begin{align*}
p_{K_4^{1234}}&=x_{3,4} x_{1,2}^2+x_{3,4}^2 x_{1,2}+x_{1,3} x_{2,3} x_{1,2}-x_{1,4} x_{2,3} x_{1,2}-x_{1,3} x_{2,4} x_{1,2}\\
&+x_{1,4}^2 x_{2,3}+x_{1,3} x_{2,4}^2+x_{1,4} x_{2,4} x_{1,2}-
x_{1,3} x_{3,4} x_{1,2}-x_{1,4} x_{3,4} x_{1,2}\\
&+x_{1,3}^2 x_{2,4}+x_{1,4} x_{2,3}^2-x_{2,3} x_{3,4} x_{1,2}-x_{2,4} x_{3,4} x_{1,2}
+x_{2,3} x_{2,4} x_{3,4}\\
&-x_{1,3} x_{2,4} x_{3,4}-x_{1,3} x_{1,4} x_{2,3}
-x_{1,3} x_{1,4} x_{2,4}-x_{1,3} x_{2,3} x_{2,4}\\
&-x_{1,4} x_{2,3} x_{2,4}+x_{1,3} x_{1,4} x_{3,4}-x_{1,4} x_{2,3} x_{3,4}.
\end{align*}
This polynomial has 22 terms, its homogeneous degree is 3, and it is of degree 2 in any of its variables.

\subparagraph{Wheel on 4 vertices.} The circuit polynomial $p_{W^{1234,5}}$ supported on the wheel $W^{1234,5}$ on vertices $1234$ with $5$ in the centre is the resultant $\res{p_{K_4^{1245}}}{p_{K_4^{2345}}}{x_{2,4}}$, where $p_{K_4^{1245}}$ and $p_{K_4^{2345}}$ are circuit polynomials supported on $K_4$ on the vertices $1245$ and $2345$, respectively. It can be verified with a computer algebra package that this resultant is irreducible.


This polynomial has 843 terms, its homogeneous degree is 8, and it is of degree 4 in each of its variables.

\subparagraph{The ``2D double banana''.} The circuit polynomial supported on the 2D double banana $\{12,13,23,24,34,15,16,56,45,46\}$ is the resultant $\res{p_{K_4^{1234}}}{p_{K_4^{1456}}}{x_{1,4}}$, where $p_{K_4^{1234}}$ and $p_{K_4^{1456}}$ are circuit polynomials supported on $K_4$ on the vertices $1245$ and $2345$, respectively. It can be verified with a computer algebra package that this resultant is irreducible.

This polynomial has 1752 terms, its homogeneous degree is 8, and it is of degree 4 in each of its variables.

\subparagraph{Wheel on 5 vertices.} The circuit polynomial $p_{W^{12345,6}}$ supported on the wheel $W^{12345,6}$ on vertices $12345$ with $6$ in the centre is the resultant $\res{p_{W^{1245,6}}}{p_{K_4^{2346}}}{x_{2,4}}$, where $p_{W^{1245,6}}$ and $p_{K_4^{2346}}$ are circuit polynomials supported on the wheel $W^{1245,6}$ on the four vertices $1245$ with 6 in the centre and the $K_4$ on te vertices $2345$, respectively. It can be verified with a computer algebra package that this resultant is irreducible.

This polynomial has 273123 terms, its homogeneous degree is 20, and it is of degree 8 in each of its variables.

\subparagraph{The Desargues-plus-one circuit.}\label{example:desarguesPlus1} Let $D$ be the Desargues graph $\{12,14,15,23,26,34,36,45$, $56\}$. The graph $D$ can be completed to a circuit by adjoining to it exactly one of the edges in the set $\{13,16,24,25,35,46\}$, however all choices result in isomorphic graphs. We will therefore only show how to obtain the circuit polynomial $p_{D\cup\{16\}}$ as the other circuit polynomials can be obtained by appropriate relabeling.

The circuit polynomial $p_{D\cup\{16\}}$ is the resultant $\res{p_{W^{1234,5}}}{p_{K_4^{1245}}}{x_{2,5}}$. Its irreducibility can be verified with a computer algebra package. It has 658175 terms, its homogeneous degree is 20, it is of degree 12 in the variable $x_{1,6}$ and of degree $8$ in the remaining variables.

\subparagraph{The $K_{3,3}$-plus-one circuit.}\label{example:K33plus1} Consider the complete bipartite graph $K_{3,3}$ on the vertex partition $\{1,4,5\}\cup\{2,3,6\}$. It can be completed to a circuit by adjoining to it exactly one of the edges in the set $\{14,15,45,23,26,36\}$, say $45$, and as in the case of the Desargues-plus-one circuit, any other choice will result in a graph isomorphic to $K_{3,3}\cup\{45\}$.

The circuit polynomial $p_{K_{3,3,}\cup\{45\}}$ is an irreducible polynomial with 1018050 terms, of homogeneous degree 18, and of degree 8 in each of its variables.

\begin{proposition}
	\label{prop:k33notObtainable}The circuit polynomial $p_{K_{3,3,}\cup\{45\}}$ can not be obtained directly as the resultant of two circuit polynomials supported on circuits with 6 or less vertices.
\end{proposition}

Here by directly we mean that no choice of two circuit polynomials on at most 6 vertices will give an irreducible resultant supported on ${K_{3,3}\cup\{45\}}$.

\begin{proof}Let $p_1$ and $p_2$ be any two circuit polynomials supported on a circuit with 6 or less vertices. Consider the resultant $\res{p_1}{p_2}{x_e}$ for some common variable $x_e$. Let $h_1$ and $h_2$ be the homogeneous degrees, and let $d_1$ and $d_2$ be the degrees in $x_e$ of $p_1$ and $p_2$, respectively.

By Proposition \ref{prop:resultantHomogeneous}, the homogeneous degree of $\res{p_1}{p_2}{x_e}$ is $h_1d_2+h_2d_1-d_1d_2$, so if $\res{p_1}{p_2}{x_e}=p_{K_{3,3,}\cup\{45\}}$, then $h_1d_2+h_2d_1-d_1d_2=18$. However, the only possible choices for $(h_i,d_i)$ are $(3,2)$, $(8,4)$, $(18,8)$, $(20,8)$ and $(20,12)$ (say which graph they correspond to), none of which result with $h_1d_2+h_2d_1-d_1d_2=18$.\end{proof}

However, $K_{3,3}\cup\{45\}$ is the combinatorial resultant of the wheels $W^{1234,5}$ and $W^{1346,5}$ on 1234 with 5 in the centre, and on 1346 with 5 in the centre, respectively, with the edge 15 eliminated. Since $p_{W^{1234,5}}$ and $p_{W^{1346,5}}$ are of homogeneous degree 8, and have degree $4$ in $x_{1,5}$, it follows from Proposition \ref{prop:resultantHomogeneous} that $\res{p_{W^{1234,5}}}{p_{W^{1346,5}}}{x_{1,5}}$ has homogeneous degree 48, hence $p_{K_{3,3}\cup\{45\}}$ appears in $\res{p_{W^{1234,5}}}{p_{W^{1346,5}}}{x_{1,5}}$ as a factor, with multiplicity not greater than 2.

We were not able to compute the resultant $\res{p_{W^{1234,5}}}{p_{W^{1346,5}}}{x_{1,5}}$ before our machines ran out of memory. We are currently exploring a High Performance Computer platform to test the limits of our method. Meanwhile, we developed an extension of our algorithm \cite{malic:streinu:k33} and computed $p_{K_{3,3}\cup\{45\}}$ with it.

\section{Concluding Remarks}
\label{sec:concludingRemarks}

In this paper we introduced the combinatorial resultant operation, analogous to the  classical resultant of polynomials.  To demonstrate the effectiveness of our method we conclude by listing in Table~\ref{tbl:circPolyAppendix} the circuit polynomials that we could compute within a reasonable amount of time. The most challenging was the $K_{3,3}$-plus-one circuit, which required an extension of the method presented here: this {\em extended resultant} is the topic of an upcoming paper \cite{malic:streinu:k33}.

However, this method still has computational drawbacks in the sense that it requires an irreducibility check, with a possible further factorization and an ideal membership test for those factors that have the support of a circuit.

Ideally we would like to detect combinatorially when a resultant of two circuit polynomials that has the support of a circuit will be irreducible. The absolute irreducibility test of Gao \cite{Gao} which states that a polynomial is absolutely irreducible if and only if its Newton polytope is integrally indecomposable, in conjunction with the description of the Newton polytope of the resultant of two polynomials by Gelfand, Kapranov and Zelevinsky \cite{GelfandKapranovZelevinskyNewton, GelfandKapranovZelevinsky} gives a combinatorial criterion for absolute irreducibility, but not for irreducibility over $\mathbb Q$. However, not every circuit polynomial is absolutely irreducible, for example the circuit polynomial of a wheel on 4 vertices is irreducible over $\mathbb Q$ but not absolutely irreducible. 

\subparagraph{What we observed in practice.} It is worthwhile to note that whenever in our computations we had to decide which factor of a resultant belonged to $\cm n$ we never had to preform an ideal membership test. It was always sufficient to inspect only the supports of the irreducible factors of the resultant, and in all cases we had the situation where all but one irreducible factor were supported on Laman graphs, and one factor was supported on a dependent set. It seems unlikely that this is the general case, and it would be of interest to determine under which conditions does the resultant have exactly one factor (up to multiplicity) supported on a dependent set in $\amat{\cm n}$. 

We conclude the paper with a list of open problems about the algebraic and geometric structure of the resultant of two circuit polynomials.

\begin{problem}
	Let $C$, $A$ and $B$ be sparsity circuits such that $C=\cres{A}{B}{e}$. Let $p_C$, $p_{A}$ and $p_{B}$ be the corresponding circuit polynomials and assume that $\res{p_{A}}{p_{B}}{x_e}$ is reducible.

Under which conditions on $C$, $A$, $B$, $p_{A}$ and $p_{B}$ does $\res{p_{A}}{p_{B}}{x_e}$ have up to multiplicity exactly one irreducible factor over $\q$ equal to $p_C$?\end{problem}

\begin{problem}
More generally, if $p,q\in \cm n$ with $x_e\in\supp p\cap\supp q$, under which conditions on $p$, $q$, $\supp p$ and $\supp q$ does $\res{p}{q}{x_e}$ have up to multiplicity exactly one irreducible factor supported on a dependent set in $\amat{\cm n}$?
\end{problem}

The irreducible factor of $\res{p_{A}}{p_{B}}{x_e}$ corresponding to the circuit polynomial $p_C$ defines the variety in $\mathbb C^{\binom{n}{2}}$ of complex placements of $C$, however it is not clear what is the geometric significance of other factors, if there are any.

\begin{problem}If $\res{p_{A}}{p_{B}}{x_e}$ has a factor supported on a Laman graph, what is its significance? Does this factor have a geometric interpretation?
\end{problem}

The degree with respect to a single variable in the support of a circuit polynomial is bounded by the number of complex placements of the underlying Laman graph \cite{streinu:borcea:numberEmbeddings:2004}, with an upper bound of $\binom{2n-4}{n-2}\approx 4^n$. We have observed that the number of terms of circuit polynomials quickly becomes large, as shown in the Table \ref{tbl:circPolyAppendix}. 

\begin{problem}
How big do circuit polynomials get, i.e.\ what are the upper and lower bounds on the number of monomial terms relative to the number of vertices $n$?\end{problem}

\begin{table}[ht]
\centering
\caption{Results: all circuit polynomials on $n\leq 6$ vertices and two circuit polynomials on $n=7$ vertices. For the definition of Extended Resultant, see  \cite{malic:streinu:k33}.}
\label{tbl:circPolyAppendix}
\begin{tabularx}{\textheight}
	{|c|c|c|>{\centering}m{1.65cm}|c|>{\centering}m{0.95cm}|}
	\cline{1-6}
	
	$n$ & 
	Circuit & Method & 
	Comp.\ time (seconds) &
	No.\ terms & 
	Hom.\ degree \cr	
	
	\cline{1-6}	
	\multirow{2}*{4} & 
	\multirow{2}*{$K_4$} & 
	\multirow{2}*{Determinant} & 
	\multirow{2}*{0.0008} & 
	\multirow{2}*{22} & 
	\multirow{2}*{3} \cr	
	& & & & & 
	\cr	
		
	\cline{1-6}			
	\multirow{2}*{5} &  
	\multirow{2}*{Wheel on 4 vertices} & 
	\grobner{} & 
	0.02 &
	\multirow{2}*{843} & 
	\multirow{2}*{8} \cr	
	& & Resultant & 0.013 & & 
	\cr	
			
	\cline{1-6}			
	\multirow{2}*{6} & 
	\multirow{2}*{2D double banana} & 
	\grobner{} & 
	0.164 & 
	\multirow{2}*{1 752} & 
	\multirow{2}*{8} \cr	
	& & Resultant & 0.029 & & 
	\cr	
		
	\cline{1-6}			
	\multirow{2}*{6} & 
	\multirow{2}*{Wheel on 5 vertices} & 
	\grobner{} & 
	10 857 &
	\multirow{2}*{273 123} & 
	\multirow{2}*{20} \cr	
	& & Resultant & 7.07 & & 
	\cr	
	
	\cline{1-6}		
	\multirow{2}*{6} & 
	\multirow{2}*{Desargues-plus-one} & 
	\grobner{} & 454 753 & 	
	\multirow{2}*{658 175} & 
	\multirow{2}*{20} \cr	
	& & Resultant & 14.62 & & 
	\cr	
		
	\cline{1-6}					
	\multirow{2}*{6} & 
	\multirow{2}*{$K_{3,3}$-plus-one} & 
	\multirow{2}*{Extended Resultant} & 
	\multirow{2}*{1 402} & 
	\multirow{2}*{1 018 050} & 
	\multirow{2}*{18} \cr	
	& & & & & 
	\cr	
		
	\cline{1-6}			
	\multirow{2}*{7} & 
	\multirow{2}*{2D double banana $\oplus_{16}$ $K_4^{1567}$} & 
	\multirow{2}*{Resultant} & 
	\multirow{2}*{38.14} & \multirow{2}*{1 053 933} & 
	\multirow{2}*{20} \cr	
	& & & & & 
	\cr	
			
	\cline{1-6}		
	\multirow{2}*{7} & 
	\multirow{2}*{2D double banana $\oplus_{56}$ $K_4^{4567}$} & 
	\multirow{2}*{Resultant} & 
	\multirow{2}*{89.86} & 
	\multirow{2}*{2 579 050} & 
	\multirow{2}*{20} \cr
	& & & & &
	\cr		
		
	\cline{1-6}
	
\end{tabularx}
\end{table}

\bibliography{biblio} 

\end{document}